\documentclass[12pt]{amsart}
\usepackage{epsfig}
\usepackage{comment}
\usepackage{mathtools}
\usepackage[subnum]{cases}
\usepackage{enumerate}
\usepackage{bbm,bm}
\usepackage{amsmath,amssymb,mathrsfs}
\usepackage{color}
\usepackage[bookmarks=true,
bookmarksnumbered=true, breaklinks=true,
pdfstartview=FitH, hyperfigures=false,
plainpages=false, naturalnames=true,
colorlinks=true,pagebackref=true,
pdfpagelabels,colorlinks=true,
	linkcolor=blue,
	citecolor=black,
	filecolor=black,
	urlcolor=black]{hyperref}
\usepackage{wasysym}

\usepackage{pgfplots}
\pgfplotsset{width=10cm,compat=1.9}
\usepgfplotslibrary{fillbetween}

\usepgfplotslibrary{external}
\newtheorem{theorem}{Theorem}[section]
\newtheorem{definition}[theorem]{Definition}
\newtheorem{lemma}[theorem]{Lemma}
\newtheorem{corollary}[theorem]{Corollary}
\newtheorem{remark}[theorem]{Remark}

\newtheorem{proposition}[theorem]{Proposition}

\newcommand{\R}{\mathbb{R}} 
\DeclareMathOperator{\LC}{LC} 
\newcommand{\Sp}{\mathbb{S}^{n-1}}

\DeclareMathOperator{\dom}{dom}

\DeclareMathOperator{\vol}{vol}

\DeclareMathOperator{\lev}{lev}

\DeclareMathOperator{\supp}{supp}
\DeclareMathOperator{\ent}{Ent}

\title[Blaschke operations on log-concave functions]{Blaschke operations on log-concave functions \\and affine isoperimetric inequalities}

\author{Effrosyni Chasioti and Steven Hoehner}
\date{\today}

\begin{document}

\setcounter{footnote}{0}
\maketitle

\begin{abstract}\noindent
We introduce  Blaschke addition and homothety operations on log-concave functions and study their affine-geometric consequences. Our starting point is the first variation formula of Falah and Rotem ({\it Calc. Var. and PDE}, 2026), which associates to each log-concave function a pair of surface area measures. Using the additivity of these measures, we define a canonical Blaschke sum and Blaschke homothety on the class of log-concave functions, uniquely determined up to translation. We establish the basic algebraic properties of these operations, define the associated Blaschke symmetral, and show that this symmetrization preserves both total mass and the first quermassintegral. We also prove that successive  Blaschke symmetrizations converge, after translations, to a radially symmetric log-concave function, which we call the mean Blaschke symmetral.

We then relate the canonical theory to projection-type constructions. In particular, we show that the functional projection body arising from the first variation coincides with the projection body of the asymmetric LYZ body, and we derive corresponding intertwining properties. As applications, we prove concavity of the entropy with respect to the canonical Blaschke sum and obtain associated Kneser--S\"uss-type inequalities. 

We also study a functional version of affine surface area, and prove affine isoperimetric inequalities for log-concave functions. In particular, we obtain a Blaschke-concavity property for the affine surface area and show that it is maximized, under fixed first quermassintegral, by radially symmetric  functions.
\end{abstract}

\renewcommand{\thefootnote}{}

\renewcommand{\thefootnote}{\arabic{footnote}}
\setcounter{footnote}{0}

\tableofcontents

\section {Introduction and main results}

Blaschke addition is one of the fundamental operations in the Brunn--Minkowski theory. For convex bodies, it is defined by addition of surface area measures and plays a central role in the study of affine isoperimetric inequalities, mixed volumes, and projection bodies. In contrast with Minkowski addition, Blaschke addition is nonlinear at the level of support functions and is not monotone with respect to set inclusion. Nevertheless, it enjoys remarkable geometric structure. Its associated symmetrization procedure, the Blaschke symmetral, preserves surface area and interacts naturally with many affine-geometric extremal problems; see, for example,  \cite{BucurFragalaLamboley,Lutwak-1991,segal2014}.

Over the last several years, many classical constructions from convex geometry have been extended from convex bodies to log-concave functions. This functional viewpoint has led to powerful analogues of mixed volumes, projection bodies, affine surface area, and geometric inequalities, see, for example,  \cite{BCF-2014,Caglar-et-al-2016,Milman-Rotem-alpha,Milman-Rotem,Rotem2012,Rotem2013,Rotem-2022}. A recent advance in this direction is the first variation formula of Falah and Rotem \cite{Falah-Rotem} (stated in Theorem \ref{thm:rotem-falah-variation} below) which associates to each log-concave function $f=e^{-\varphi}$ a pair of measures $(\mu_f,\nu_f)$ that serve as functional surface area measures and characterize the first variation of the total mass. Combined with their solution of the corresponding functional Minkowski problem in \cite{Falah-Rotem} (stated in Theorem \ref{thm:rotem-falah-main} below), this provides a natural framework in which one may ask whether Blaschke-type operations can be defined directly on log-concave functions.

The first goal of this paper is to show that the answer to this question is affirmative. Using the additivity of the surface area measures $(\mu_f,\nu_f)$, we define a canonical  Blaschke sum and Blaschke homothety on the class $\LC_n$ of coercive log-concave functions by requiring additivity of the surface area measures. The functional Minkowski problem implies that the resulting function exists and is unique up to translation. This gives a direct analogue of classical Blaschke addition in the log-concave setting.

Our first main result establishes the basic structure of these operations. We prove that the functional Blaschke sum is commutative and associative up to translation, that the corresponding homothety has the expected linearity properties, and that the first variation is additive in the Blaschke variable. We then define the functional Blaschke symmetral of a log-concave function  $f\in\LC_n$ about the hyperplane $H=u^\perp$, $u\in\Sp$, by
\[
B_u^\sharp f := \left(\tfrac12 \odot f\right)\sharp \left(\tfrac12 \odot R_u f\right),
\]
and show that it is reflection-invariant up to translation, idempotent up to translation, and invariant on $H$-symmetric functions. Moreover, it preserves both total mass and the first quermassintegral $W_1$. Thus the functional symmetrization behaves, in the log-concave setting, much like the classical Blaschke symmetral behaves for convex bodies.

A second theme of the paper is the interaction between the canonical Blaschke theory and functional projection body constructions. For a log-concave function $f$, the first variation against the indicator of a segment defines a natural support function, and hence an origin-symmetric convex body. We show that this body is precisely the projection body of the asymmetric LYZ body associated with $f$. This leads to a function-valued projection body operator and a corresponding polar projection body function. We show that these constructions intertwine naturally with the canonical Blaschke operations, and provide the correct functional analogue of the classical relation between Blaschke addition and projection bodies.

Our next main result concerns convergence of successive functional  Blaschke symmetrizations. We prove that every log-concave function can be symmetrized along a suitable sequence of directions so that, after translations, the resulting sequence hypo-converges to a radially symmetric log-concave function. We call this limit the \emph{mean Blaschke symmetral} of $f$ and denote it by $f^\sharp$. This result provides the canonical radial representative associated with the Blaschke symmetrization process, and is the key tool for the affine isoperimetric applications developed later in the paper.

The functional Blaschke sum also has several analytic consequences. By combining the functional Minkowski first inequality with the Blaschke additivity of the first variation, we prove that the entropy is concave with respect to Blaschke convex combinations. This yields functional and geometric Kneser--S\"uss-type inequalities and shows that the canonical Blaschke structure carries nontrivial information, even though the total mass itself turns out to be additive under $\sharp$. We also show that the functional Blaschke sum is not monotone, which explains why one cannot define a satisfactory Blaschke addition on log-concave functions by applying the classical body-valued operation directly to superlevel sets.

Inspired by the extended affine surface area studied by Lutwak \cite{Lutwak-1991}, we study a functional affine surface area $\Omega_\sharp(f)$ defined by an infimum over convex test bodies through the first variation. We show that $\Omega_\sharp$ is affinely covariant, upper semicontinuous with respect to cosmic convergence of the surface area measures, and homogeneous of degree $\frac{n}{n+1}$. Furthermore, we prove that $\Omega_\sharp$ is $\frac{n+1}{n}$-concave with respect to the canonical Blaschke addition. The mean Blaschke symmetral then yields an affine isoperimetric comparison principle:
\[
\Omega_\sharp(f)\leq \Omega_\sharp(f^\sharp),
\]
with equality for radially symmetric  functions. Since $f^\sharp$ is radially symmetric and preserves the first quermassintegral $W_1$, we obtain a sharp affine isoperimetric inequality for log-concave functions. We also introduce the associated geominimal surface area, and prove the corresponding comparison inequality in this setting.

The paper is organized as follows. In Section \ref{sec:Blaschke-sum}, we introduce the canonical Blaschke sum and Blaschke symmetrization on log-concave functions, establish their basic properties, relate them to the asymmetric LYZ body, prove convergence of successive canonical symmetrizations, and derive new entropy inequalities. Section \ref{sec:affine-surface-area} develops the functional affine surface area and proves the corresponding affine isoperimetric inequalities. 
Finally, the Appendix contains the proofs of some of our technical results, including the convergence theorem for canonical Blaschke symmetrizations and the computation of the affine surface area of the mean Blaschke symmetral.


\subsection{Background and notation}

We write $\mathcal K^n$ for the class of convex bodies in $\R^n$, that is, convex compact sets with nonempty interior. We let $\mathcal K^n_o$  be  the subclass of bodies containing the origin in
their interiors. A convex body $K\in\mathcal{K}_o^n$ is \emph{centrally symmetric} if for all $x\in K$, we have $-x\in K$, and we let $\mathcal{K}_{\rm cs}^n$ denote the class of  centrally symmetric convex bodies. The Euclidean unit ball is denoted by $B_2^n$, its boundary sphere by $\Sp$, and $\omega_n := \vol_n(B_2^n)$. If $K \in \mathcal K^n$, then $h_K$ denotes its support
function,
\[
h_K(x)=\sup_{y\in K}\langle x,y\rangle,\qquad x\in \R^n.
\]
For $K\in\mathcal{K}_o^n$, we let $K^\circ$ denote the polar body
\[
K^\circ=\{x\in \R^n:\langle x,y\rangle\leq 1 \text{ for all } y\in K\}.
\]

We write $S_K$ for the surface area measure of $K$ on $\Sp$ and $S(K)=S_K(\Sp)$ for the surface area of $K$. For convex bodies $K,L\in\mathcal K^n$, their \emph{Blaschke sum} $K\#L$ is the convex body, unique up to translation, characterized by
\[
S_{K\#L}=S_K+S_L.
\]
For $\lambda>0$, the \emph{Blaschke scalar multiplication} is given by
\[
\lambda\cdot K=\lambda^{\frac1{n-1}}K,
\]
so that $S_{\lambda\cdot K}=\lambda S_K$. Blaschke addition is commutative and associative. However, unlike Minkowski addition, it is not monotone with respect to set inclusion. The \emph{Kneser--S\"uss inequality} states that the volume functional is $\frac{n-1}{n}$-concave with respect to the Blaschke operations, that is, for all $K,L\in\mathcal{K}^n$ and all $\lambda\in[0,1]$, we have
\begin{equation}\label{eqn:classical-KS-ineq}
    \vol_n(\lambda\cdot K\#(1-\lambda)\cdot L)^{\frac{n-1}{n}} \geq \lambda\vol_n(K)^{\frac{n-1}{n}}+(1-\lambda)\vol_n(L)^{\frac{n-1}{n}}.
\end{equation}
Equality holds if and only if $K$ and $L$ are homothetic, see, e.g., \cite[Theorem 8.2.3]{SchneiderBook}. 

Given $K\in\mathcal{K}^n$ and $u\in\Sp$, the \emph{classical Blaschke symmetral} of $K$ with respect to the hyperplane $u^\perp$ is
\[
B_u K = \left(\tfrac{1}{2}\cdot K\right)\#\left(\tfrac{1}{2}\cdot R_u K\right)
\]
where $R_u K$ is the reflection of $K$ about $u^\perp$. The classical Blaschke symmetral preserves surface area, is idempotent, quermassintegral-nondecreasing, and invariant on $u^\perp$-symmetric bodies and under reflections about $u^\perp$. Moreover, any convex body $K\in\mathcal{K}^n$ can be transformed into a Euclidean ball with the same surface area as $K$ using a convergent sequence of Blaschke symmetrizations (see, e.g., \cite{segal2014}). For more background on Blaschke operations, we refer the reader to, e.g., 
 \cite{symm-in-geom,gardner2006geometric,Gardner-Hug-Weil,gardner-et-al-Blaschke-2014,SchneiderBook}.

Let $K\in\mathcal{K}^n$. The \emph{projection body} $\Pi K$ of $K$ is the centrally symmetric convex set defined by 
\[
h_{\Pi K}(u)=\vol_{n-1}(K|u^\perp),\quad u\in\Sp,
\]
where $K|u^\perp$ is the projection of $K$ onto the hyperplane $u^\perp$. By Cauchy's surface area formula,
\[
h_{\Pi K}(u)=\frac{1}{2}\int_{\Sp}|\langle u,v\rangle|\,dS_K(v),\quad u\in\Sp.
\]
Note that $\Pi K\in\mathcal{K}_{\rm cs}$. The \emph{polar projection body} of $K$ is $\Pi^\circ K:=(\Pi K)^\circ$. The projection body operator and Blaschke sum are related by the following property (see, e.g., \cite[p. 83]{gardner2006geometric}):
\begin{equation}\label{eqn:projection-Blaschke-relation-classical}
    \forall K,L\in\mathcal{K}_{\rm cs}^n,\quad \Pi(K\# L)=\Pi K+\Pi L.
\end{equation}

Given $K\in\mathcal{K}^n$ and $u\in\Sp$, the \emph{Minkowski symmetral} of $K$ is defined by $\tau_u K=\tfrac{1}{2}K+\tfrac{1}{2}R_u K$. Note that by \eqref{eqn:projection-Blaschke-relation-classical}, the Blaschke and Minkowski symmetrals are related via the projection body operator as follows (see, e.g., \cite[Equation (1.7)]{segal2014}). For any $K,L\in\mathcal{K}_{\rm cs}^n$, we have
\begin{align*}
    \tau_u(\Pi K)&=\tfrac{1}{2}\Pi K+\tfrac{1}{2}R_u(\Pi K)=\tfrac{1}{2}\Pi K+\tfrac{1}{2}\Pi(R_u K)=\tfrac{1}{2}(\Pi K+\Pi (R_u K))\\
    &=\tfrac{1}{2}\Pi(K\# R_u K)=\Pi\left(\tfrac{1}{2}\cdot(K\# R_u K)\right)=\Pi(B_u K).
\end{align*}

For a function $f:\R^n\to\R$ and $u\in\Sp$, the \emph{reflection function} $R_u f$ of $f$ about the hyperplane $u^\perp$ is defined by $R_u f:=f\circ R_u$, that is,
\begin{equation}
    R_u f(x)=f(R_u x)
\end{equation}
for $x\in\R^n$, where $R_u x=x-2\langle x,u\rangle u$ is the reflection of $x$ about $u^\perp$. We say that $f$ is \emph{even} if $f(x)=f(-x)$ for all $x$, and we say that $f$ is \emph{radial}, or \emph{radially symmetric}, if it is
invariant under all reflections, that is, 
\[
\forall u\in\Sp,\quad f=R_u f.
\]
Equivalently, a radial function depends only on $|x|$.

For a nonnegative function $f:\R^n\to[0,\infty)$, we write
\[
\supp(f)=\overline{\{x\in\R^n:f(x)>0\}}
\]
for its \emph{support}, and for $t>0$ we define its \emph{superlevel set} by
\[
\lev_{\geq t} f := \{x\in\R^n:f(x)\geq t\}.
\]
For measurable $f:\R^n\to[0,\infty)$, the \emph{total mass} of $f$ is denoted by $J(f)=\int_{\R^n}f(x)\,dx$, provided the integral exists. In this case, the total mass can be expressed via the \emph{layer-cake formula}
\[
J(f)=\int_{\R^n} f(x)\,dx=\int_0^\infty \vol_n(\lev_{\geq t} f)\,dt.
\]


Let $\psi:\R^n\to(-\infty,+\infty]$ be a convex function. The \emph{domain} of $\psi$ is $\dom(\psi)=\{x\in\R^n:\, \psi(x)<\infty\}$. We say that  $\psi$ is \emph{proper} if $\psi\not\equiv+\infty$, i.e., if $\dom(\psi)\neq\varnothing$, and we say that $\psi$ is \emph{coercive} if $\psi(x)\to\infty$ as $|x|\to\infty$. Let $\mathrm{Conv}(\R^n)$ denote the class of all convex functions $\psi:\R^n\to(-\infty,+\infty]$ which are proper, coercive, lower semicontinuous, and satisfy $\operatorname{int}(\dom(\psi))\neq\varnothing$.

Given a convex function $\psi\in\mathrm{Conv}(\R^n)$, let
$\mathcal{L}\psi(x)=\sup_{y\in\R^n} (\langle x,y \rangle - \psi(y))$,  $x\in\R^n$, denote the \textit{Legendre--Fenchel transform} of $\psi$. By \cite[Theorem 1.6.13]{RockafellarBook}, since  $\psi\colon \R^n\to (-\infty,\infty]$ is a proper, lower semicontinuous, convex function, the Legendre--Fenchel transform $\mathcal{L}\psi$ is a proper, lower semicontinuous, convex function on $\R^n$. Moreover, it is an involution, i.e., $\mathcal{L}\mathcal{L} \psi = \psi$, and it is order-reversing, i.e., if $\psi\leq\varphi$, then $\mathcal{L}\psi\geq\mathcal{L}\varphi$. Note that by \cite[Theorem 11.8]{Rockafellar-Wets}, a proper, lower semicontinuous, convex function $\psi: \R^n\to (-\infty,\infty]$ is coercive if and only if $o\in \operatorname{int}(\dom(\mathcal{L} \psi))$.

A function $f:\R^n\to[0,\infty)$ is \emph{logarithmically concave} (or \emph{log-concave}) if $\log f$ is concave. Every upper semicontinuous log-concave function may be written as $f=e^{-\phi}$, where $\phi:\R^n\to(-\infty,+\infty]$ is convex. In this paper, we will work with the class $\LC_n=\{f=e^{-\varphi}:\,\varphi\in\mathrm{Conv}(\R^n)\}$ of coercive log-concave functions on $\R^n$, and the subclass $\LC_n^e=\{f\in\LC_n:\,f\text{ is even}\}$ of even log-concave functions. Note that every $f\in \LC_n$ is integrable (see, e.g., \cite{cordero-erasquin-klartag}), proper, and upper semicontinuous. 

The indicator function of $K\in\mathcal{K}^n$ is defined, for $x\in\R^n$, by
\[
\mathbbm{1}_K(x)=\begin{cases}
    1&\quad\text{if }x\in K,\\
    0&\quad\text{if }x\not\in K.
\end{cases}
\]
In particular, if $K\in\mathcal{K}^n$, then $\mathbbm{1}_K\in\LC_n$ and $J(\mathbbm{1}_K)=\vol_n(K)$. 

The \emph{support function} of $f=e^{-\varphi}\in\LC_n$ is given by
\[
h_f = \mathcal{L}(-\log f)=\mathcal{L}\varphi.
\]
This extends the classical support function of a convex body in the sense that if $K\in\mathcal K^n$, then $h_{\mathbbm{1}_K}=h_K$.

For $f=e^{-\varphi}$ and $g=e^{-\psi}$ belonging to $\LC_n$, we use the standard Asplund sum (supremal convolution) operations
\[
f\star g = e^{-(\varphi \square \psi)},
\qquad
(\lambda\cdot f)(x)=e^{-\lambda\varphi(x/\lambda)}, \qquad \lambda>0,
\]
where $\varphi\square\psi$ is the infimal convolution
\[
(\varphi\square\psi)(x)=\inf_{x=x_1+x_2}\{\varphi(x_1)+\psi(x_2)\}, \qquad x\in\R^n.
\]
Equivalently,
\[
(\lambda\cdot f)(x)=f(x/\lambda)^\lambda .
\]

We say that a sequence $\{\varphi_k\}_{k=1}^\infty\subset\mathrm{Conv}(\R^n)$ \emph{epi-converges} to $\varphi\in\mathrm{Conv}(\R^n)$ if for every $x\in\R^n$, we have $\liminf_{k\to\infty}\varphi_k(x_k)\geq \varphi(x)$ for every sequence $\{x_k\}_{k=1}^\infty\subset\R^n$ converging to $x$, and $\limsup_{k\to\infty}\varphi_k(x_k)\leq \varphi(x)$ for some sequence $\{x_k\}_{k=1}^\infty\subset\R^n$ converging to $x$. A symmetric notion of convergence is defined for log-concave functions. We say that a sequence $\{f_k\}_{k=1}^\infty\subset\LC_n$ \emph{hypo-converges} to $f\in\LC_n$ if $\limsup_{k\to\infty}f_k(x_k)\leq f(x)$ for every sequence $\{x_k\}_{k=1}^\infty\subset\R^n$ converging to $x$, and $\liminf_{k\to\infty}f_k(x_k)\geq f(x)$ for some sequence $\{x_k\}_{k=1}^\infty\subset\R^n$ converging to $x$. Note that the sequence $\{\varphi_k\}_{k=1}^\infty\subset\mathrm{Conv}(\R^n)$ epi-converges to $\varphi\in\mathrm{Conv}(\R^n)$ if and only if $\{f_k=e^{-\varphi_k}\}_{k=1}^\infty$ hypo-converges to $f=e^{-\varphi}\in\LC_n$. For more background on convex and log-concave functions, we refer the reader to \cite{Colesanti-inbook,RockafellarBook,Rockafellar-Wets}.

We say that a function $\xi:\R^n\to\R$ is \emph{cosmically continuous} if:
\begin{itemize}
    \item $\xi$ is continuous in the usual sense on $\R^n$.

    \item The limit $\overline{\xi}(\theta)=\lim_{\lambda\to\infty}\frac{\xi(\lambda\theta)}{\lambda}$ exists (in the finite sense) \emph{uniformly} in $\theta\in\Sp$.
\end{itemize}
This definition was given by Falah and Rotem in  \cite[Definition 1.7]{Falah-Rotem}. For a cosmically continuous function $\xi:\mathbb R^n\to\mathbb R$, write
\[
\overline{\xi}(\theta):=\lim_{\lambda\to\infty}\frac{\xi(\lambda\theta)}{\lambda},
\qquad \theta\in\mathbb S^{n-1}.
\]
We will also need the following definition, which is from \cite[Definition 1.8]{Falah-Rotem}. Let $\{\mu_k\}_{k=1}^\infty$ and $\mu$ be finite Borel measures on $\R^n$, and let $\{\nu_k\}_{k=1}^\infty$ and $\nu$ be finite Borel measures on $\Sp$. We say that $(\mu_k,\nu_k)\to(\mu,\nu)$ \emph{cosmically} if, for every cosmically continuous $\xi$, we have
\[
\int_{\R^n}\xi\,d\mu_k+\int_{\Sp}\overline{\xi}\,d\nu_k
\stackrel{k\to\infty}{\longrightarrow}\int_{\R^n}\xi\,d\mu+\int_{\Sp}\overline{\xi}\,d\nu.
\]



Throughout the paper, whenever an object is defined through surface area measures in the log-concave setting, uniqueness is understood up to translations, in accordance with the
functional Minkowski problem (see Theorem \ref{thm:rotem-falah-main}).

\section{Canonical definition of the functional Blaschke sum}\label{sec:Blaschke-sum}

Let $f,g\in\LC_n$. The \emph{first variation} $\delta(f,g)$ of the total mass
$\int f$ along the perturbation $f\star (t\cdot g)$ is defined by (see \cite{Colesanti-Fragala-variational,Rotem-2022})
\begin{equation}
    \delta(f,g)=\lim_{t\to 0^+}\frac{\int(f\star(t\cdot g))-\int f}{t}.
\end{equation}

The following result is due to Falah and Rotem \cite[Theorem 1.4]{Falah-Rotem}.
\begin{theorem}\label{thm:rotem-falah-variation}
    For all $f\in\LC_n$ and every upper semicontinuous log-concave function $g:\R^n\to\R$, we have 
    \begin{equation}
        \delta(f,g)=\int_{\R^n}h_g\,d\mu_f+\int_{\Sp}h_{\supp(g)}\,d\nu_f.
    \end{equation}
    Here $\mu_f$ is the Borel measure on $\R^n$ defined by $\mu_f=(\nabla\phi)_\sharp (f\,dx)$ and $\nu_f$ is the Borel measure on $\Sp$ defined by $\nu_f=(n_{\supp(f)})_\sharp (f\,d\mathcal{H}^{n-1}|_{\partial\supp(f)})$. We refer to the pair $(\mu_f,\nu_f)$ as the \emph{surface area measures of $f$}.
\end{theorem}

\begin{remark}\label{rmk:SA}
    Let $f\in\LC_n$ and $g=\mathbbm{1}_{B_2^n}$. Then $h_{\mathbbm{1}_{B_2^n}}(x)=|x|$ and $h_{\supp(\mathbbm{1}_{B_2^n})}\equiv 1$. Hence, the \emph{first quermassintegral of $f$}, denoted $W_1(f)$, may be defined by
\begin{equation}\label{def:quermassintegral}
    W_1(f):=\delta(f,\mathbbm{1}_{B_2^n})=\int_{\R^n}|x|\,d\mu_f(x)+\nu_f(\Sp).
\end{equation}
Let $K\in\mathcal{K}^n$ and set $f=\mathbbm{1}_K$. Since $\mu_{\mathbbm{1}_K}=(\nabla\varphi_K)_{\sharp}(\mathbbm{1}_K\,dx)=\vol_n(K)\delta_o$, we have  
\[
\int_{\R^n}|x|\,d\mu_{\mathbbm{1}_K}(x)=\int_{\R^n}|x|\,d(\vol_n(K)\delta_o)(x)=\vol_n(K)|o|=0.
\]
Moreover, $\nu_{\mathbbm{1}_K}(\Sp)=S(K)$, so \eqref{def:quermassintegral} recovers the surface area $S(K)$:
\[
W_1(\mathbbm{1}_K)=\delta(\mathbbm{1}_K,\mathbbm{1}_{B_2^n})=S(K).
\]
\end{remark}

The following level set formula connects the surface area $W_1(f)$ to the first variation.

\begin{lemma}\label{thm:SA-layer-cake}
Let $f:\mathbb R^n\to[0,\infty)$ be upper semicontinuous and quasiconcave, and assume that
for every $s>0$ the superlevel set $\lev_{\geq s}f$ is a compact convex set, possibly empty. For a compact convex set $K$, let
\[
S(K):=\left.\frac{d}{dt}\right|_{t=0+}\vol_n(K+tB_2^n)
\]
denote its Minkowski surface area, or, equivalently, the coefficient of $t$ in Steiner's formula. Assume that there exists $t_0>0$ such that $F(t_0)<\infty$, where
\[
F(t):=\int_{\mathbb R^n}(f\star \mathbbm{1}_{tB_2^n})(x)\,dx
=\int_0^\infty \vol_n((\lev_{\geq s}f)+tB_2^n)\,ds .
\]
Then $F$ is right-differentiable at $t=0$, and
\[
\left.\frac{dF}{dt}\right|_{t=0+}=\int_0^\infty S(\lev_{\geq s}f)\,ds.
\]
In particular, if $f\in\LC_n$, then the finiteness condition $F(t_0)<\infty$ holds for every
$t_0>0$, and
\[
W_1(f)=\delta(f,\mathbbm{1}_{B_2^n})
=\int_0^\infty S(\lev_{\geq s}f)\,ds.
\]
\end{lemma}
Since we could not find a convenient reference in exactly this form, we include a proof in the appendix.

\vspace{2mm}

Falah and Rotem \cite[Theorem 1.5]{Falah-Rotem} also gave a complete solution of the functional Minkowski problem:

\begin{theorem}\label{thm:rotem-falah-main}
    Let $\mu$ be a finite Borel measure on $\R^n$ and $\nu$ be a finite Borel measure on $\Sp$. Then there exists a log-concave function $f\in\LC_n$ such that $(\mu_f,\nu_f)=(\mu,\nu)$ if and only if the pair $(\mu,\nu)$ satisfies the following conditions:
    \begin{itemize}
        \item[(i)] $\mu$ is not identically 0.

        \item[(ii)] $\mu$ has finite first moment, and the measure $\mu+\nu$ is centered in the sense that for every $\theta\in\Sp$ we have
        \[
\int_{\R^n}\langle x,\theta\rangle\,d\mu(x)+\int_{\Sp}\langle x,\theta\rangle\, d\nu(x)=0.
        \]

        \item[(iii)] $\mu$ and $\nu$ are not supported on a common hyperplane.
    \end{itemize}
    In this case, the function $f$ is unique up to translations: If $g\in\LC_n$ also satisfies $(\mu_g,\nu_g)=(\mu,\nu)$, then there exists $v\in\R^n$ such that $g(x)=f(x+v)$.
\end{theorem}

    A direct application of this result yields the following 

\begin{lemma}\label{lem:minkowski-blaschke-sum}
    Let $f_1,f_2\in\LC_n$. Define the finite Borel measure $\mu_{f_1\sharp f_2}:=\mu_{f_1}+\mu_{f_2}$ on $\R^n$ and the finite Borel measure $\nu_{f_1\sharp f_2}:=\nu_{f_1}+\nu_{f_2}$ on $\Sp$.   The pair $(\mu_{f_1\sharp f_2},\nu_{f_1\sharp f_2})$ defines the surface area measures of a log-concave function $h\in\LC_n$. In this case, the function $h$ is unique up to translations: if $\tilde{h}\in\LC_n$ also satisfies $(\mu_{\tilde{h}},\nu_{\tilde{h}})=(\mu_{f_1\sharp f_2},\nu_{f_1\sharp f_2})$, then there exists $v\in\R^n$ such that $\tilde{h}(x)=h(x+v)$.
\end{lemma}

\begin{proof}
We verify that the pair $(\mu_{f_1\sharp f_2},\nu_{f_1\sharp f_2})$ satisfies the three conditions in Lemma \ref{thm:rotem-falah-main}. For part (i), note that by Theorem \ref{thm:rotem-falah-main}, $\mu_{f_1}$ and $\mu_{f_2}$ are not identically 0, so the sum $\mu_{f_1}+\mu_{f_2}=\mu_{f_1\sharp f_2}$ is not identically 0.

(ii)  By Theorem \ref{thm:rotem-falah-main},  the measures $\mu_{f_1}$ and $\mu_{f_2}$ have finite first moment, i.e., 
$\int_{\R^n} |x|\,d\mu_{f_i}(x)<\infty$ for $i=1,2$, so their sum does as well. Indeed, by the linearity of integration with respect to measures,
\[
\int_{\R^n} |x|\,d\mu_{f_1\sharp f_2}(x)
=\int_{\R^n} |x|\,d(\mu_{f_1}+\mu_{f_2})(x)
=\int_{\R^n} |x|\,d\mu_{f_1}(x)+\int_{\R^n} |x|\,d\mu_{f_2}(x)
<\infty.
\]
To see the centering condition, note that for any $\theta\in\Sp$ we have
\begin{align*}
    \int_{\R^n}\langle x,\theta\rangle\,d\mu_{f_i}(x)+\int_{\Sp}\langle x,\theta\rangle\, d\nu_{f_i}(x)=0\qquad i=1,2.
\end{align*}
Hence by the linearity of the integral,
\begin{align*}
    &\int_{\R^n}\langle x,\theta\rangle\,d\mu_{f_1\sharp f_2}(x)+\int_{\Sp}\langle x,\theta\rangle\, d\nu_{f_1\sharp f_2}(x)\\
    =&\int_{\R^n}\langle x,\theta\rangle\,d(\mu_{f_1}+\mu_{f_2})(x)+\int_{\Sp}\langle x,\theta\rangle\, d(\nu_{f_1}+\nu_{f_2})(x)\\
    =&\int_{\R^n}\langle x,\theta\rangle\,d\mu_{f_1}(x)+\int_{\Sp}\langle x,\theta\rangle\, d\nu_{f_1}(x)+\int_{\R^n}\langle x,\theta\rangle\,d\mu_{f_2}(x)+\int_{\Sp}\langle x,\theta\rangle\, d\nu_{f_2}(x)=0.
\end{align*}

(iii) By Theorem \ref{thm:rotem-falah-main}, each pair $(\mu_{f_1},\nu_{f_1}),(\mu_{f_2},\nu_{f_2})$ is not supported on a hyperplane. Suppose by way of contradiction that there exists a hyperplane $H\subset\R^n$ such that
$\mu_1+\mu_2$ is supported on $H$ and $\nu_1+\nu_2$ is supported on $H\cap \Sp$. Then
\[
0=(\mu_1+\mu_2)(\R^n\setminus H)=\mu_1(\R^n\setminus H)+\mu_2(\R^n\setminus H),
\]
and since the measures are nonnegative this implies $\mu_1(\R^n\setminus H)=\mu_2(\R^n\setminus H)=0$,
i.e., $\mu_1$ and $\mu_2$ are each supported on $H$. Similarly, the identity
\[
0=(\nu_1+\nu_2)(\Sp\setminus (H\cap \Sp))
=\nu_1(\Sp\setminus (H\cap \Sp))+\nu_2(\Sp\setminus (H\cap \Sp))
\]
implies that $\nu_1$ and $\nu_2$ are each supported on $H\cap \Sp$.
Hence for each $i=1,2$, both $\mu_i$ and $\nu_i$ are supported on the same hyperplane $H$, a contradiction.
\end{proof}

This shows that the following Blaschke sum and homothety operations on log-concave functions are well-defined and unique up to translations:

\begin{definition}[Functional Blaschke sum and homothety]
Let $f_1,f_2\in\LC_n$. Define the \emph{canonical Blaschke sum} $f_1\sharp f_2\in\LC_n$ to be the log-concave function whose surface area measures are the sums of the surface area measures of $f_1,f_2$, i.e., 
\[
(\mu_{f_1\sharp f_2},\nu_{f_1\sharp f_2})=(\mu_{f_1}+\mu_{f_2},\nu_{f_1}+\nu_{f_2}).
\]
For $f\in\LC_n$ and $\lambda>0$, we define the \emph{Blaschke homothety} $\lambda\odot f$ by 
\[
(\mu_{\lambda\odot f},\nu_{\lambda\odot f})=(\lambda\mu_f,\lambda\nu_f).
\]
\end{definition}

Next, we highlight some of the basic properties of these operations.

\begin{proposition}\label{prop:Blaschke-operations}
    Let $f_1,f_2,f_3,f\in\LC_n$, and let $g:\R^n\to\R$ be upper semicontinuous and log-concave. The following algebraic properties hold: 
    \begin{itemize}
        \item[(i)] $f_1\sharp f_2\in\LC_n$ and $\lambda\odot f\in\LC_n$ for every $\lambda>0$.
        
        \item[(ii)] $\delta(f_1\sharp f_2,g)=\delta(f_1,g)+\delta(f_2,g)$.

        \item[(iii)]  $\delta(\lambda\odot f,g)=\lambda\delta(f,g)$ for every $\lambda>0$.

        \item[(iv)] $f_1\sharp f_2=f_2\sharp f_1$ (up to a translation).

        \item[(v)] $(f_1\sharp f_2)\sharp f_3=f_1\sharp(f_2\sharp f_3)$ (up to a translation).

        \item[(vi)] $1\odot f=f$ (up to a translation).

        \item[(vii)] For all  $\lambda_1,\lambda_2>0$, we have $(\lambda_1\odot f)\sharp(\lambda_2\odot f)=(\lambda_1+\lambda_2)\odot f$ (up to a translation).

        \item[(viii)] We have $J(f_1\sharp f_2)=J(f_1)+J(f_2)$ and $J(\lambda\odot f)=\lambda J(f)$.
    \end{itemize}
\end{proposition}

\begin{proof}
(i) The statement $f_1\sharp f_2\in\LC_n$ follows from Lemma \ref{lem:minkowski-blaschke-sum}. An argument similar to the proof of Lemma \ref{lem:minkowski-blaschke-sum} shows that the pair $(\mu_{\lambda\odot f},\nu_{\lambda\odot f})$ satisfies the conditions of Theorem  \ref{thm:rotem-falah-main} and hence are the surface area measures of a log-concave function $\lambda\odot f\in\LC_n$.

(ii) By Theorem \ref{thm:rotem-falah-variation} and the definition of the Blaschke sum,
\begin{align*}
\delta(f_1\sharp f_2,g)&=\int_{\R^n}h_g\,d\mu_{f_1\sharp f_2}+\int_{\Sp}h_{\supp(g)}\,d\nu_{f_1\sharp f_2}\\
&=\int_{\R^n}h_g\,d\mu_{f_1}+\int_{\Sp}h_{\supp(g)}\,d\nu_{f_1}+\int_{\R^n}h_g\,d\mu_{f_2}+\int_{\Sp}h_{\supp(g)}\,d\nu_{f_2}\\
&=\delta(f_1,g)+\delta(f_2,g).
\end{align*}

(iii) By definition of the Blaschke $\lambda$-homothety, for every $\lambda>0$ we have 
\begin{align*}
    \delta(\lambda\odot f,g) &=\int_{\R^n}h_g\,d\mu_{\lambda\odot f}+\int_{\Sp}h_{\supp(g)}\,d\nu_{\lambda\odot f}
    =\int_{\R^n}h_g\,d(\lambda\mu_f)+\int_{\Sp}h_{\supp(g)}\,d(\lambda\nu_f)\\
    &=\lambda\int_{\R^n}h_g\,d\mu_f+\lambda\int_{\Sp}h_{\supp(g)}\,d\nu_f=\lambda\delta(f,g).
\end{align*}

Parts (iv) and (v) follow immediately  from the definition of the Blaschke sum, and part (vi) follows immediately from the definition of $\lambda\odot f$ with $\lambda=1$. Part (vii) can be proved along the same lines as the previous parts.

    (viii) Note that for any $f\in\LC_n$, we have  $\mu_f(\R^n)=\int_{\R^n} f(x)\,dx$. Thus by definition of the Blaschke sum, 
    \[
J(f_1\sharp f_2)=\int f_1\sharp f_2=\mu_{f_1\sharp f_2}(\R^n)=\mu_{f_1}(\R^n)+\mu_{f_2}(\R^n)=\int f_1+\int f_2=J(f_1)+J(f_2).
    \]
    Similarly, if $\lambda>0$ then
    \[
J(\lambda\odot f)=\int\lambda\odot f=\mu_{\lambda\odot f}(\R^n)=\lambda\mu_f(\R^n)=\lambda\int f=\lambda J(f).
    \]
\end{proof}

\begin{remark}
Note that (viii) shows we cannot expect a direct linear analogue of the classical Kneser--S\"uss inequality \eqref{eqn:classical-KS-ineq}.
\end{remark}
\subsection{Functional Blaschke symmetral}

\begin{definition}\label{def:mainDef}
    Let $f\in\LC_n$ and $u\in\Sp$. The \emph{Blaschke symmetrization $B^\sharp_u f$ of $f$ with respect to the hyperplane $u^\perp$} is defined by
    \[
    B^\sharp_u f:=\left(\frac{1}{2}\odot f\right)\sharp\left(\frac{1}{2}\odot R_u f\right).
    \]
    In other words, $B_u^\sharp f$ is the log-concave function whose surface area measures are
    \[
(\mu_{B_u^\sharp f},\nu_{B_u^\sharp f})=\frac{1}{2}\left(\mu_f+(R_u)_\#\mu_f,\nu_f+(R_u)_\#\nu_f\right).
    \]
\end{definition}

By Proposition \ref{prop:Blaschke-operations}(i), the Blaschke symmetral is well-defined and $B_u^\sharp f\in\LC_n$. It is unique up to translations in the sense of Theorem \ref{thm:rotem-falah-main}.

In parallel with the geometric theory, we highlight some basic properties of the functional Blaschke symmetral. A key feature is that it preserves surface area and total mass.

\begin{proposition}\label{prop-Blaschke-symmetral}
    Let $f\in\LC_n$ and $u\in\Sp$. 
    \begin{itemize}
        \item[(i)] (Reflection invariance) $R_u(B_u^\sharp f)=B_u^\sharp f$ (up to a translation).

        \item[(ii)] (Idempotence) $B_u^\sharp(B_u^\sharp f)=B_u^\sharp f$ (up to a translation).

        \item[(iii)] (Invariance on $H$-symmetric functions) If $f=R_u f$, then $B_u^\sharp f=f$ (up to a translation).

        \item[(iv)] (Surface area preserving) $W_1(B_u^\sharp f)=W_1(f)$. 

        \item[(v)] (Total mass preserving) $J(B_u^\sharp f)=J(f)$.
        
        \end{itemize}
\end{proposition}

\begin{proof}
    (i) Let $h=e^{-\psi}\in\LC_n$ and set $\widetilde{h}:=R_u h=e^{-\widetilde{\psi}}$ where $\widetilde{\psi}=R_u\psi$. Note that  $R_u$ is a linear orthogonal transformation and $R_u^{-1}=R_u$. Since $\psi$ is convex, it is locally Lipschitz on $\operatorname{int}(\dom(\psi))$, and therefore is differentiable a.e. by Rademacher's theorem (see, e.g., \cite[Theorem 25.5]{RockafellarBook}). Thus, for a.e. $x$ such that $\psi$ is differentiable at $R_u x$, the chain rule yields
\[
\nabla\widetilde{\psi}(x)=\nabla(\psi\circ R_u)(x)=R_u^\top\nabla\psi(R_u x)=R_u\nabla\psi(R_u x).
\] 
Now fix a Borel set $A\subset\R^n$. By the definition of $\mu_{\widetilde{h}}$,
    \[
\mu_{\widetilde{h}}(A)=\int_{\{x: \nabla\widetilde{\psi}(x)\in A\}}\widetilde{h}(x)\,dx=\int_{\{x: R_u\nabla\psi(R_u x)\in A\}}h(R_u x)\,dx.
    \]
    Change variables to $y=R_u x$. Since $R_u$ is orthogonal, $dx=dy$. Hence
    \[
\mu_{\widetilde{h}}(A)=\int_{\{y: R_u\nabla\psi(y)\in A\}}h(y)\,dy
=\int_{\{y: \nabla\psi(y)\in R_u A\}}h(y)\,dy=\mu_h(R_u A).
    \]
    Since $R_u^{-1}=R_u$, this means that $\mu_{R_u h}=(R_u)_{\#}\mu_h$. Similarly, since $\supp(R_u h)=R_u(\supp(h))$, and since reflections send outer unit normals to outer unit normals, we obtain 
    \[
n_{R_u\supp(h)}(R_u y)=R_u n_{\supp(h)}(y)\quad \text{for }\mathcal{H}^{n-1}\text{-a.e. }y\in\supp(h),
    \]
    and since $R_u$ preserves the measure $\mathcal{H}^{n-1}$, we get that for every Borel set $B\subset\Sp$, we have $\nu_{R_u h}(B)=\nu_h(R_u B)$. Hence $\nu_{R_u h}=(R_u)_{\#}\nu_h$. Therefore, for every $h\in\LC_n$, we get 
    \begin{equation}\label{eq:SA-measures-reflection}
(\mu_{R_u h},\nu_{R_u h})=((R_u)_{\#}\mu_h, (R_u)_{\#}\nu_h).
    \end{equation}

    Applying \eqref{eq:SA-measures-reflection} to $h=B_u^\sharp f$ where $f\in\LC_n$, we obtain
    \[
\mu_{R_u(B_u^\sharp f)}=(R_u)_{\#}\mu_{B_u^\sharp f}=\frac{1}{2}\left((R_u)_{\#}\mu_f+(R_u)_{\#}(R_u)_{\#}\mu_f\right).
    \]
    Since $R_u^2=\mathrm{Id}$, we have $(R_u)_{\#}(R_u)_{\#}\mu_f=\mu_f$, so
    \[
\mu_{R_u(B_u^\sharp f)}=\frac{1}{2}\left((R_u)_{\#}\mu_f+\mu_f\right)=\mu_{B_u^\sharp f}.
    \]
    Following along the same lines, we also obtain
    \[
\nu_{R_u(B_u^\sharp f)}=(R_u)_{\#}\nu_{B_u^\sharp f}=\frac{1}{2}\left((R_u)_{\#}\nu_f+\nu_f\right)=\nu_{B_u^\sharp f}.
    \]
    Therefore,
    \begin{equation}
        (\mu_{R_u(B_u^\sharp f)},\nu_{R_u(B_u^\sharp f)})=(\mu_{B_u^\sharp f},\nu_{B_u^\sharp f}).
    \end{equation}
Finally, by Lemma \ref{lem:minkowski-blaschke-sum}, since $B_u^\sharp f$ and $R_u(B_u^\sharp f)$ have the same pair of surface area measures, there exists $v\in\R^n$ such that $R_u(B_u^\sharp f)(x)=B_u^\sharp f(x+v)$ for all $x\in\R^n$. Equivalently, $R_u(B_u^\sharp f)=B_u^\sharp f$ up to translation.

(ii) Let $g=B_u^\sharp f$. By (i), $R_u g$ is a translate of $g$. Since translations do not change surface area measures, it follows that $(\mu_{R_u g},\nu_{R_u g})=(\mu_g,\nu_g)$. On the other hand, reflection pushes these measures forward, so $(\mu_{R_u g},\nu_{R_u g})=((R_u)_{\#}\mu_g,(R_u)_{\#}\nu_g)$. Hence $(R_u)_{\#}\mu_g=\mu_g$ and $(R_u)_{\#}\nu_g=\nu_g$. Thus by definition of the Blaschke sum, this implies
\[
(\mu_{B_u^\sharp g},\nu_{B_u^\sharp g})=\frac{1}{2}(\mu_g+(R_u)_{\#}\mu_g,\nu_g+(R_u)_{\#}\nu_g)=\frac{1}{2}(\mu_g+\mu_g,\nu_g+\nu_g)=(\mu_g,\nu_g).
\]
Therefore, $B_u^\sharp g$ and $g$ have the same surface area measures. By Lemma \ref{lem:minkowski-blaschke-sum}, this implies that $B_u^\sharp g$ and $g$ are equal up to a translation.

(iii) Assume that $f=R_u f$. By Proposition \ref{prop:Blaschke-operations} (parts (vii) and (vi)),
\[
B_u^\sharp f=\left(\tfrac{1}{2}\odot f\right)\sharp\left(\tfrac{1}{2}\odot R_u f\right)=\left(\tfrac{1}{2}\odot f\right)\sharp\left(\tfrac{1}{2}\odot  f\right)=\left(\tfrac{1}{2}+\tfrac{1}{2}\right)\odot f=1\odot f=f.
\]

    (iv) Since reflection preserves $|x|$ and total mass, we have
    \[
\int_{\R^n}|x|\,d\mu_{B_u^\sharp f}(x)=\int_{\R^n}|x|\,d\mu_f(x)\quad\text{and}\quad \nu_{B_u^\sharp f}(\Sp)=\nu_f(\Sp).
    \]
    Thus, by Remark \ref{rmk:SA} we get $W_1(B_u^\sharp f)=W_1(f)$.

    Part (v) follows from Proposition \ref{prop:Blaschke-operations}(viii) with $f_1=f$, $f_2=R_u f$ and $\lambda=1/2$.
\end{proof}

\begin{remark}
    Combining Proposition \ref{prop-Blaschke-symmetral}(iv) and Remark \ref{rmk:SA}, in the case $f=\mathbbm{1}_K$ we recover the identity $S(B_u K)=S(K)$.
\end{remark}

\begin{proposition}\label{prop:indicators}
    Let $K$ and $L$ be convex bodies in $\R^n$, and set $\lambda:=\frac{\vol_n(K)+\vol_n(L)}{\vol_n(K\# L)}$. Then, up to translation,
    \[
\mathbbm{1}_K\sharp \mathbbm{1}_L=\lambda^{-(n-1)}\mathbbm{1}_{\lambda(K\# L)}.
    \]
\end{proposition}

\begin{proof}
Let $f=e^{-\phi_K}=\mathbbm{1}_K$ where $K$ is a convex body in $\R^n$. Then $\phi_K$ is the (convex) characteristic function of $K$, defined by $\phi_K(x)=0$ if $x\in K$, and $\phi_K(x)=+\infty$ if $x\not\in K$. Hence $\nabla\phi_K=0$ a.e. on $K$, and
\[
\mu_{\mathbbm{1}_K}=(\nabla\phi)_\#(\mathbbm{1}_K\,dx)=\vol_n(K)\delta_o,
\]
where $\delta_o$ is the point mass distribution at $o$.

Moreover, $\supp(\mathbbm{1}_K)=K$ and $\mathbbm{1}_K\equiv 1$ on $\partial K$, so
\[
\nu_{\mathbbm{1}_K}=(n_K)_\#(\mathcal{H}^{n-1}|_{\partial K})=S_K
\]
is the usual surface area measure on $K$. Therefore,
\begin{equation}\label{eq:indicator-SA-measures}
(\mu_{\mathbbm{1}_K\sharp\mathbbm{1}_L},\nu_{\mathbbm{1}_K\sharp\mathbbm{1}_L})=((\vol_n(K)+\vol_n(L))\delta_o,S_K+S_L).
\end{equation}

Now let $\lambda$ be defined as in the statement of the proposition, and set $c:=\lambda^{-(n-1)}$. Consider the log-concave function $h:=c\mathbbm{1}_{\lambda(K\# L)}$. Its surface area measures are:
\begin{equation}
\begin{aligned}\label{eq:SA-measures-h}
       \mu_h&=c\vol_n(\lambda(K\# L))\delta_o=(\vol_n(K)+\vol_n(L))\delta_o,\\ \nu_h&=cS_{\lambda(K\# L)}=c\lambda^{n-1}S_{K\# L}= S_K+S_L.
\end{aligned}
\end{equation}
From \eqref{eq:indicator-SA-measures} and \eqref{eq:SA-measures-h}, we derive that $\mathbbm{1}_K\sharp\mathbbm{1}_L$  and $h$ have the same surface area measures. Thus by Theorem \ref{thm:rotem-falah-main}, the functions agree up to translations. 
\end{proof}

The next proposition shows that the functional Blaschke addition is not monotone, i.e., the condition $f\leq g$ does not guarantee that there is a choice of representatives of the translation classes of $f\sharp h$ and $g\sharp h$ for which $f\sharp h\leq g\sharp h$ pointwise. In what follows, given a function $f\in\LC_n$, we let $[f]$ denote the equivalence class of all translations of $f$, i.e.,
\[
[f]:=\{g\in\LC_n:\,\exists v\in\R^n\text{ s.t. }g(x)=f(x+v)\text{ for all }x\in\R^n\}.
\]
Obviously, $f\in[f]$ so $[f]$ is always nonempty.

\begin{proposition}\label{prop:nonmonotonicity}
There exist $f,g,h\in\LC_n$ with $f\leq g$ such that for no representatives $F\in[f\sharp h]$ and $G\in[g\sharp h]$ do we have $F\leq G$.
\end{proposition}

\begin{proof}
Let $f=h=\mathbbm{1}_{B_2^n}$ and $g=\mathbbm{1}_{2B_2^n}$. Note that $f,g,h\in\LC_n$, and $f\leq g$ since $B_2^n\subset 2B_2^n$. Since $S_{rB_2^n}=r^{n-1}S_{B_2^n}$ for every $r>0$, we have $rB_2^n\# sB_2^n=(r^{n-1}+s^{n-1})^{\frac{1}{n-1}}B_2^n$ for all $r,s>0$. Hence $B_2^n\# B_2^n=2^{\frac{1}{n-1}}B_2^n$ and $2B_2^n\# B_2^n=(2^{n-1}+1)^{\frac{1}{n-1}}B_2^n$. Now we apply Proposition \ref{prop:indicators}. For $f\sharp h=\mathbbm{1}_{B_2^n}\sharp\mathbbm{1}_{B_2^n}$ we obtain
\[
\lambda_1:=\frac{\vol_n(B_2^n)+\vol_n(B_2^n)}{\vol_n(B_2^n\# B_2^n)}=\frac{2\vol_n(B_2^n)}{2^{\frac{n}{n-1}}\vol_n(B_2^n)}=2^{-\frac{1}{n-1}}.
\]
Thus 
\[
f\sharp h=\lambda_1^{-(n-1)}\mathbbm{1}_{\lambda_1(B_2^n\# B_2^n)}=2\mathbbm{1}_{B_2^n}.
\]

For $g\sharp h=\mathbbm{1}_{2B_2^n}\sharp \mathbbm{1}_{B_2^n}$, we obtain
\[
\lambda_2:=\frac{\vol_n(2B_2^n)+\vol_n(B_2^n)}{\vol_n(2B_2^n\# B_2^n)}=\frac{(2^n+1)\vol_n(B_2^n)}{(2^{n-1}+1)^{\frac{n}{n-1}}\vol_n(B_2^n)}=\frac{2^n+1}{(2^{n-1}+1)^{\frac{n}{n-1}}}.
\]
Hence
\[
g\sharp h=\lambda_2^{-(n-1)}\mathbbm{1}_{\lambda_2(2B_2^n\# B_2^n)}=c_n\mathbbm{1}_{r_n B_2^n}
\]
where $c_n:=\frac{(2^{n-1}+1)^n}{(2^n+1)^{n-1}}$ and $r_n:=\frac{2^n+1}{2^{n-1}+1}$.

By convexity, for fixed $p>1$ and all $a,b\geq 0$, we have $(a+b)^p\leq 2^{p-1}(a^p+b^p)$ with equality if and only if $a=b$. Using this with $a=2^{n-1}$, $b=1$ and $p=\frac{n}{n-1}>1$, we get $(2^{n-1}+1)^{\frac{n}{n-1}}<2^{\frac{1}{n-1}}(2^n+1)$; the inequality here is strict because $a\neq b$. Thus, $c_n<2$. Now let $F\in[f\sharp h]$ and $G\in[g\sharp h]$. Since translations do not change the supremum, we have $\sup_{\R^n}F=2$ and $\sup_{\R^n}G<2$. Thus, $F\leq G$ is impossible, because it would imply $2=\sup F\leq\sup G<2$, a contradiction. Therefore, for $f=\mathbbm{1}_{B_2^n}, g=\mathbbm{1}_{2B_2^n}, h=\mathbbm{1}_{B_2^n}\in\LC_n$ with $f\leq g$, there do not exist representatives $F\in[f\sharp h]$ and $G\in[g\sharp h]$ such that $F\leq G$.
\end{proof}

\begin{remark}
    Proposition \ref{prop:nonmonotonicity} shows that Blaschke addition is not a generalized supremal convolution as defined in \cite{MMRR26}. In particular, the transference principle therein cannot be used to deduce Theorem~\ref{thm:entropy-concave} or the geometric inequality that it implies.
\end{remark}

\subsection{Relation to the ``functional" projection body/asymmetric LYZ body}

In view of the intertwining property \eqref{eqn:projection-Blaschke-relation-classical} of the projection body and Blaschke sum operations of convex bodies, it is natural to seek a functional analogue. As we will show in this section, the \emph{asymmetric LYZ body} is a natural choice. The following definition is from \cite{Langharst-Marin-Ulivelli} (see also \cite{Fang-Zhou-Advances}).

\begin{definition}[Asymmetric LYZ body]
    Let $f\in\LC_n$. Then the \emph{asymmetric LYZ body} $\langle f\rangle\subset\R^n$ of $f$ is the unique convex body with center of mass at the origin with the following property: for every 1-homogeneous function $F:\R^n\setminus\{o\}\to\R$,
    \begin{equation}\label{eq:LYZ-body}
        \int_{\partial\langle f\rangle}F(n_{\langle f\rangle}(y))\,d\mathcal{H}^{n-1}(y)=\int_{\R^n}F(-\nabla f(x))\,dx+\int_{\partial\supp(f)}F(n_{\supp(f)}(y))f(y)\,d\mathcal{H}^{n-1}(y).
    \end{equation}
\end{definition}

As shown in \cite{Langharst-Marin-Ulivelli}, when $F=\frac{1}{2}|\langle\theta,\cdot\rangle|$, applying the polar projection body operator $\Pi^\circ$, we get that $\Pi^\circ\langle f\rangle$ is the unit ball of the norm 
\begin{equation}\label{eq:norm}
\|v\|_{\Pi^\circ\langle f\rangle}
=\frac12\int_{\mathbb R^n}|\langle \nabla f(x),v\rangle|\,dx
+\frac12\int_{\partial\operatorname{supp}(f)}
|\langle n_{\operatorname{supp}(f)}(y),v\rangle|f(y)\,d\mathcal H^{n-1}(y). 
\end{equation}

To see how this connects to the Blaschke sum, let $f\in\LC_n$, and for $u\in\Sp$, set $g_u=\mathbbm{1}_{[-u,u]}$. Consider the ``functional" projection body $\Pi_\sharp^b f$ defined by its support function:
\begin{equation}\label{eq:functional-proj-body}
    h_{\Pi_\sharp^b f}(u):=\frac{1}{2}\delta(f,\mathbbm{1}_{[-u,u]}),\qquad u\in\Sp.
\end{equation}
Using Theorem \ref{thm:rotem-falah-variation} with $g=g_u$, and the fact that for an indicator of a segment we have $h_{g_u}(x)=h_{\supp(g_u)}(x)=|\langle x,u\rangle|$, definition \eqref{eq:functional-proj-body} becomes
\begin{equation}\label{fn-proj-body-2}
    h_{\Pi_\sharp^b f}(u)=\frac{1}{2}\int_{\R^n}|\langle z,u\rangle|\,d\mu_f(z)+\frac{1}{2}\int_{\Sp}|\langle \theta,u\rangle|\,d\nu_f(\theta).
\end{equation}
This is finite, even, positively 1-homogeneous, and convex, so it is sublinear, and therefore it is the support function of an origin-symmetric convex body. Using $\mu_f=(\nabla\phi)_\#(f\,dx)$ and $\nu_f=(n_{\supp(f)})_\#(f\,\mathcal{H}^{n-1}|_{\partial\supp(f)})$, this becomes 
\begin{equation}\label{fn-proj-body-3}
    h_{\Pi_\sharp^b f}(u)=\frac{1}{2}\int_{\R^n}|\langle \nabla\phi(x),u\rangle|f(x)\,dx+\frac{1}{2}\int_{\partial\supp(f)}|\langle n_{\supp(f)}(y),u\rangle|f(y)\,d\mathcal{H}^{n-1}(y).
\end{equation}
Therefore, up to notation and with the choice of $F$ above, the projection body defined in \eqref{eq:functional-proj-body} is a special case of the  asymmetric LYZ body:
\begin{equation}
    \Pi_\sharp^b f=\Pi\langle f\rangle.
\end{equation}
We use the notation $\Pi_\sharp^b$ to highlight that the specific function $F=\frac{1}{2}|\langle\theta,\cdot\rangle|$ has been chosen.

    One may ask: Why not just define a log-concave function $\Pi_\sharp f$ whose support function $h_f=\mathcal{L}(-\log f)$ is given by $\frac{1}{2}\delta(f,\mathbbm{1}_{[-u,u]})$ (or, equivalently, by \eqref{fn-proj-body-2})? This support function is bijective and additive \cite{Rotem2013} with respect to the Asplund sum, i.e., $h_{f_1\star f_2}=h_{f_1}+h_{f_2}$ for all $f_1,f_2\in\LC_n$, and it extends the classical body support function in the sense that $h_{\mathbbm{1}_K}=h_K$. So if we want a function $\Pi_\sharp f$ whose support function is $h_{\Pi_\sharp f}=h_{\Pi_\sharp^b f}$, then by the injectivity of the support map (see \cite{Rotem2013}), there is exactly one such log-concave function. Since $h_{\Pi_\sharp^b f}$ is the support function of a convex body, the unique inverse is 
\[
\Pi_\sharp f=\mathbbm{1}_{\Pi_\sharp^b f}=\mathbbm{1}_{\Pi\langle f\rangle}.
\]

\noindent The function-valued operator is then given in
\begin{definition}\label{def:fn-proj-body}
    Let $f\in\LC_n$. 
    \begin{itemize}
        \item[(i)] The \emph{functional projection body} $\Pi_\sharp f$ of $f$ is defined by
    \[
\Pi_\sharp f:=\mathbbm{1}_{\Pi_\sharp^b f}.
    \]

        \item[(ii)] The \emph{functional polar projection body} $\Pi_\sharp^\circ f$ of $f$ is defined by
    \[
\Pi_\sharp^\circ f:=\exp\bigl(-h_{\Pi_\sharp^b f}\bigr)=\exp\left(-h_{\Pi\langle f\rangle}\right).
\]
    \end{itemize}
\end{definition}


The polar definition is justified as follows. If $g=e^{-\varphi}\in\LC_n$, then $h_g=\mathcal{L}(-\log g)=\mathcal{L}\varphi$ and the polar function is $g^\circ=e^{-h_g}$. With $g=\Pi_\sharp f$, we have $h_g=h_{\Pi_\sharp^b f}$, so $g^\circ=\exp(-h_{\Pi_\sharp^b f})$.

These projection body functions have been considered before, see, e.g., \cite{Langharst-Marin-Ulivelli} and the references therein. In the next lemma, we show that the functional projection body is well-defined. 

\begin{proposition}\label{prop:proj-body-LCn}
     For every $f\in\LC_n$, we have $\Pi_\sharp f\in\LC_n^e$ and $\Pi_\sharp^\circ f\in\LC_n^e$.
\end{proposition}
\noindent For the reader's convenience, the proof of Proposition \ref{prop:proj-body-LCn} is given in the Appendix.
\vspace{2mm}

Next, we investigate how the Blaschke sum relates to the functional projection body and the functional polar projection body. In particular, under the Blaschke symmetrization $B_u^\sharp$, the polar projection body function undergoes a  geometric-mean symmetrization.

\begin{proposition}
    Let $f_1,f_2,f\in\LC_n$, $\lambda>0$ and $u\in\Sp$. Then the following properties hold:
    \begin{itemize}
        \item[(i)] $\Pi_\sharp^b(f_1\sharp f_2)=\Pi_\sharp^b f_1+\Pi_\sharp^b f_2$, which is equivalent to
        \[
        \Pi_\sharp(f_1\sharp f_2)=\Pi_\sharp f_1\star\Pi_\sharp f_2
        \qquad\text{and}\qquad
        \Pi_\sharp^\circ(f_1\sharp f_2)=\Pi_\sharp^\circ f_1\cdot\Pi_\sharp^\circ f_2;
        \]

        \item[(ii)] $\Pi_\sharp^b(\lambda\odot f)=\lambda\,\Pi_\sharp^b f$, which is equivalent to
        \[
        \Pi_\sharp(\lambda\odot f)=\lambda\cdot\Pi_\sharp f
        \qquad\text{and}\qquad
        \Pi_\sharp^\circ(\lambda\odot f)=(\Pi_\sharp^\circ f)^\lambda;
        \]

        \item[(iii)] $\Pi_\sharp^b(B_u^\sharp f)=\tau_u(\Pi_\sharp^b f)$, which is equivalent to
        \[
        \Pi_\sharp(B_u^\sharp f)=\tau_u(\Pi_\sharp f)
        \qquad\text{and}\qquad
        \Pi_\sharp^\circ(B_u^\sharp f)=\sqrt{\Pi_\sharp^\circ f\cdot\Pi_\sharp^\circ(R_u f)};
        \]

        \item[(iv)] $J(\Pi_\sharp^\circ f)=n!\vol_n((\Pi\langle f\rangle)^\circ)$.
    \end{itemize}
\end{proposition}

\begin{proof}
    For (i), by definition \eqref{eq:functional-proj-body} and Proposition \ref{prop:Blaschke-operations}(ii), for every $u\in\Sp$ we have
    \begin{align*}
        h_{\Pi_\sharp^b(f_1\sharp f_2)}(u)
        &=\frac{1}{2}\delta(f_1\sharp f_2,\mathbbm{1}_{[-u,u]})\\
        &=\frac{1}{2}\delta(f_1,\mathbbm{1}_{[-u,u]})+\frac{1}{2}\delta(f_2,\mathbbm{1}_{[-u,u]})\\
        &=h_{\Pi_\sharp^b f_1}(u)+h_{\Pi_\sharp^b f_2}(u)
        =h_{\Pi_\sharp^b f_1+\Pi_\sharp^b f_2}(u).
    \end{align*}
    Since support functions uniquely determine convex bodies, this proves
    \[
    \Pi_\sharp^b(f_1\sharp f_2)=\Pi_\sharp^b f_1+\Pi_\sharp^b f_2.
    \]
    Therefore,
    \[
    \Pi_\sharp(f_1\sharp f_2)
    =\mathbbm{1}_{\Pi_\sharp^b(f_1\sharp f_2)}
    =\mathbbm{1}_{\Pi_\sharp^b f_1+\Pi_\sharp^b f_2}
    =\mathbbm{1}_{\Pi_\sharp^b f_1}\star \mathbbm{1}_{\Pi_\sharp^b f_2}
    =\Pi_\sharp f_1\star\Pi_\sharp f_2.
    \]
    Likewise,
    \begin{align*}
        \Pi_\sharp^\circ(f_1\sharp f_2)
        &=\exp\left(-h_{\Pi_\sharp^b(f_1\sharp f_2)}\right)
        =\exp\left(-(h_{\Pi_\sharp^b f_1}+h_{\Pi_\sharp^b f_2})\right)
        =\Pi_\sharp^\circ f_1\cdot\Pi_\sharp^\circ f_2.
    \end{align*}

    For (ii), again by definition \eqref{eq:functional-proj-body} and Proposition \ref{prop:Blaschke-operations}(iii) we obtain
    \[
    h_{\Pi_\sharp^b(\lambda\odot f)}(u)
    =\frac12\delta(\lambda\odot f,\mathbbm{1}_{[-u,u]})
    =\lambda h_{\Pi_\sharp^b f}(u),
    \]
    so $\Pi_\sharp^b(\lambda\odot f)=\lambda\,\Pi_\sharp^b f$. 
    Hence,
    \[
    \Pi_\sharp(\lambda\odot f)
    =\mathbbm{1}_{\lambda\,\Pi_\sharp^b f}
    =\lambda\cdot\mathbbm{1}_{\Pi_\sharp^b f}
    =\lambda\cdot\Pi_\sharp f
    \]
    and
    \[
    \Pi_\sharp^\circ(\lambda\odot f)
    =\exp\left(-h_{\Pi_\sharp^b(\lambda\odot f)}\right)
    =\exp\left(-\lambda h_{\Pi_\sharp^b f}\right)
    =(\Pi_\sharp^\circ f)^\lambda.
    \]

    For (iii), by Definition \ref{def:fn-proj-body} together with (i) and (ii),
    \[
    \Pi_\sharp^b(B_u^\sharp f)
    =
    \Pi_\sharp^b\!\left(\frac12\odot f\sharp \frac12\odot R_u f\right)
    =
    \frac12\,\Pi_\sharp^b f+\frac12\,\Pi_\sharp^b(R_u f).
    \]
    Moreover, for every $v\in\Sp$,
    \begin{align*}
        h_{\Pi_\sharp^b(R_u f)}(v)
        &=\frac12\,\delta(R_u f,\mathbbm{1}_{[-v,v]})
        =\frac12\,\delta(f,\mathbbm{1}_{[-R_u v,R_u v]})
        =h_{\Pi_\sharp^b f}(R_u v)
        =h_{R_u\Pi_\sharp^b f}(v),
    \end{align*}
    whence
    \[
    \Pi_\sharp^b(R_u f)=R_u\Pi_\sharp^b f.
    \]
    Therefore,
    \[
    \Pi_\sharp^b(B_u^\sharp f)
    =
    \frac12\bigl(\Pi_\sharp^b f+R_u\Pi_\sharp^b f\bigr)
    =
    \tau_u(\Pi_\sharp^b f).
    \]
    Consequently,
    \[
    \Pi_\sharp(B_u^\sharp f)
    =
    \mathbbm{1}_{\tau_u(\Pi_\sharp^b f)}
    =
    \tau_u(\Pi_\sharp f),
    \]
    and
    \begin{align*}
        \Pi_\sharp^\circ(B_u^\sharp f)
        &=\exp\left(-h_{\Pi_\sharp^b(B_u^\sharp f)}\right)=\exp\left(-\left(\tfrac{1}{2} h_{\Pi_\sharp^b f}+\tfrac{1}{2} h_{\Pi_\sharp^b(R_u f)}\right)\right)
        =\sqrt{\Pi_\sharp^\circ f\cdot\Pi_\sharp^\circ(R_u f)}.
    \end{align*}

    For (iv), recall that for every  $K\in\mathcal{K}_o^n$,
    \[
    \int_{\R^n}e^{-h_K(x)}\,dx=n!\vol_n(K^\circ).
    \]
    Applying this identity with $K=\Pi_\sharp^b f=\Pi\langle f\rangle$ yields
    \[
    J(\Pi_\sharp^\circ f)=\int_{\R^n}\Pi_\sharp^\circ f(x)\,dx
    =n!\vol_n((\Pi\langle f\rangle)^\circ).
    \qedhere
    \]
\end{proof}

\begin{remark}
   Given $f_1,f_2\in \LC_n$ and $t\in(0,1)$, set $f_t:=(1-t)\odot f_1\sharp t\odot f_2$. Set also $K_i:=\Pi\langle f_i\rangle$, $i=1,2$, and $K_t:=\Pi\langle f_t\rangle$. Applying  the Brunn--Minkowski inequality to $\Pi\langle f_t\rangle$, we obtain
    \[
J(\Pi_\sharp f_t)^{1/n}\geq (1-t)J(\Pi_\sharp f_1)^{1/n}+tJ(\Pi_\sharp f_2)^{1/n}
    \]
with equality if and only if $K_1$ and $K_2$ are homothetic; since $K_1,K_2$ are origin-symmetric, this is equivalent to $K_2=cK_1$ for some $c>0$.
\end{remark}

\begin{remark}
    Given $K\in\mathcal{K}^n$ and $u\in\Sp$, the \emph{classical Minkowski symmetrization of $K$} is defined by
    \[
    \tau_u K=\frac{1}{2}K+\frac{1}{2}R_u K.
    \]
    Since $\Pi\langle R_u f\rangle=R_u(\Pi\langle f\rangle)$, property (iii) states
    \[
    \Pi\langle B_u^\sharp f\rangle=\tau_u\Pi\langle f\rangle.
    \]
    In particular, when $f=\mathbbm{1}_K$, we recover the classical identity
    \[
    \Pi(B_u K)=\tau_u(\Pi K).
    \]
\end{remark}

\begin{remark}
    The  projection body function $\Pi_\sharp f$ intertwines the functional Blaschke sum with sup-convolution:
    \[
    \Pi_\sharp(f_1\sharp f_2)=\Pi_\sharp f_1\star\Pi_\sharp f_2.
    \]
    Likewise,
    \[
    \Pi_\sharp(B_u^\sharp f)=\left(\tfrac{1}{2}\cdot \Pi_\sharp f\right)\star\left(\tfrac{1}{2}\cdot \Pi_\sharp(R_u f)\right)=\tau_u(\Pi_\sharp f),
    \]
    where
    \[
    \tau_u g:=\left(\tfrac{1}{2}\cdot g\right)\star\left(\tfrac{1}{2}\cdot R_u g\right)
    \]
    denotes the \emph{Minkowski symmetrization} of $g\in\LC_n$, see \cite{Hoehner-2023}.
\end{remark}

   \subsection{Convergence of successive Blaschke symmetrizations}

   In this section, we prove that  any log-concave function can be made to converge--using a sequence of iterated Blaschke symmetrizations--to a radial  function which we refer to as the mean Blaschke symmetral. 

   \begin{theorem}\label{thm:Blaschke=convergence-thm}
       Let $f\in\LC_n$. Then there exists an ordered sequence of directions $u_1,u_2,\ldots\in\Sp$ and translations $v_1,v_2,\ldots\in\R^n$ such that, if $f_k:=B_{u_k}^\sharp\cdots B_{u_1}^\sharp f$, then the sequence  $\widetilde{f}_k$ defined by
       \[
\tilde{f}_k(x):=f_k(x+v_k),\quad x\in\R^n,
       \]
       hypo-converges to a radial  function $f^\sharp\in\LC_n$ as $k\to\infty$. Moreover, $f^\sharp$ is characterized, up to translation, by its surface area measures
       \[
(\mu_{f^\sharp},\nu_{f^\sharp})=\left(\int_{\mathrm{O}(n)}U_{\#}\mu_f\,d\eta(U),\int_{\mathrm{O}(n)}U_\#\nu_f\,d\eta(U)\right)
       \]
       where $\eta$ is the normalized Haar measure on $\mathrm{O}(n)$. The sequence of directions may be chosen so that the subgroup of $\mathrm{O}(n)$ generated by the reflections $R_{u_k}$ is dense in $\mathrm{O}(n)$.
   \end{theorem}

   \begin{definition}
    Given $f\in\LC_n$, we call $f^\sharp$ the \emph{mean Blaschke symmetral} of $f$.
   \end{definition}

The proof of Theorem \ref{thm:Blaschke=convergence-thm} can be found in the appendix. The next result states that the mean Blaschke symmetral preserves surface area.

\begin{proposition}\label{prop:hypo-symmetral-preserves-SA}
    For every $f\in\LC_n$, the mean Blaschke symmetral $f^\sharp$ satisfies
    \[
W_1(f)=W_1(f^\sharp)\quad\text{and}\quad J(f)=J(f^\sharp).
    \]
\end{proposition}

\begin{proof}
    Let $f\in\LC_n$. By Remark \ref{rmk:SA}, $W_1(f)=\int_{\R^n}|x|\,d\mu_f+\nu_f(\Sp)$. By Theorem \ref{thm:Blaschke=convergence-thm}, 
          \[
(\mu_{f^\sharp},\nu_{f^\sharp})=\left(\int_{\mathrm{O}(n)}U_{\#}\mu_f\,d\eta(U),\int_{\mathrm{O}(n)}U_\#\nu_f\,d\eta(U)\right).
       \]
       Hence,
       \begin{align*}
           \int_{\R^n}|x|\,d\mu_{f^\sharp} &= \int_{\mathrm{O}(n)}\int_{\R^n}|x|\,d(U_\#\mu_f)\,d\eta(U)=\int_{\mathrm{O}(n)}\int_{\R^n}|Ux|\,d\mu_f\,d\eta(U)=\int_{\R^n}|x|\,d\mu_f
       \end{align*}
       since $|Ux|=|x|$ for every $x\in\R^n$ and every $U\in\mathrm{O}(n)$. Likewise,
       \begin{align*}
           \nu_{f^\sharp}(\Sp) &=\int_{\mathrm{O}(n)}(U_\#\nu_f)(\Sp)\,d\eta(U)=\int_{\mathrm{O}(n)}\nu_f(U^{-1}\Sp)\,d\eta(U)=\nu_f(\Sp)
       \end{align*}
       since $U^{-1}\Sp=\Sp$ for every $U\in\mathrm{O}(n)$. Therefore, 
       \[
W_1(f^\sharp)=\int_{\R^n}|x|\,d\mu_{f^\sharp}+\nu_{f^\sharp}(\Sp)=\int_{\R^n}|x|\,d\mu_f+\nu_f(\Sp)=W_1(f).
       \]
       We also have
       \[
J(f^\sharp)=\mu_{f^\sharp}(\R^n)=\int_{\mathrm{O}(n)}(U_\#\mu_f)(\R^n)\,d\eta(U)=\int_{\mathrm{O}(n)}\mu_f(U^{-1}\R^n)\,d\eta(U)=\mu_f(\R^n)=J(f).
       \]
\end{proof}

\subsection{Entropy inequalities}

The following definition is from \cite[Definition 3.10]{Colesanti-Fragala-variational}.

\begin{definition}
    Let $f\in\LC_n$. The \emph{entropy} of $f$ is defined by
    \[
\ent(f)=J(f\log f)-J(f)\log J(f).
    \]
\end{definition}
\noindent Note that when $f$ is a probability density function, this reduces to $\ent(f)=\int f\log f$, which is the negative of the usual differential entropy from information theory. In particular, if $f\in\LC_n$ then by \cite[Proposition 3.11]{Colesanti-Fragala-variational},
\begin{equation}\label{eq:delta-f-f}
    \delta(f,f)=J(f)\left[n+\log J(f)\right]+\ent(f).
\end{equation}

Our goal is to prove that the entropy is concave with respect to the Blaschke sum. Our proof follows that of the classical Kneser--S\"uss inequality, which uses Minkowski's first inequality, see \cite{SchneiderBook}. In \cite{Colesanti-Fragala-variational}, they also proved the following functional analogue of Minkowski's first inequality. 

\begin{theorem}\cite[Theorem 5.1]{Colesanti-Fragala-variational}\label{thm:CF-entropy}
    For all $f,g\in\LC_n$, we have 
    \[
\delta(f,g)\geq J(f)\left[n+\log J(g)\right]+\ent(f)
    \]
    with equality if and only if $f$ and $g$ are translates, i.e., there exists $x_0\in\R^n$ such that $g(x)=f(x-x_0)$ for all $x\in\R^n$. 
\end{theorem}

For the Blaschke sum, we obtain the following version.
\begin{proposition}\label{thm:entropy}
    Let $f_1,f_2,g\in\LC_n$. Then
    \[
\delta(f_1\sharp f_2,g)\geq (J(f_1)+J(f_2))\left[n+\log J(g)\right]+\ent(f_1)+\ent(f_2)
    \]
    with equality if and only if $f_1,f_2,g$ are translates, i.e., there exist $x_1,x_2\in\R^n$ such that $g(x)=f_1(x-x_1)=f_2(x-x_2)$ for all $x\in\R^n$.
\end{proposition}

\begin{proof}
    By Proposition \ref{prop:Blaschke-operations}(ii) and Theorem \ref{thm:CF-entropy}, we obtain the inequality
    \begin{align}\label{eq:entropy-1st-variation}
        \delta(f_1\sharp f_2,g) &=\delta(f_1,g)+\delta(f_2,g)
        \geq \big(J(f_1)+J(f_2)\big)[n+\log J(g)]+\ent(f_1)+\ent(f_2).
    \end{align}
    Here we have applied Theorem \ref{thm:CF-entropy} to  the pairs $(f_1,g)$ and $(f_2,g)$. Therefore, equality holds in \eqref{eq:entropy-1st-variation} if and only if 
    \[
    \big[\exists x_1\in\R^n,\,\forall x\in\R^n, \, g(x)=f_1(x-x_1)\big]\wedge \big[\exists x_2\in\R^n,\,\forall x\in\R^n, \, g(x)=f_2(x-x_2)\big]. 
    \]
    This is logically equivalent to 
    \[
\exists x_1,x_2\in\R^n,\,\forall x\in\R^n,\, g(x)=f_1(x-x_1)=f_2(x-x_2),
    \]
    i.e., $f_1,f_2,g$ are translates.
\end{proof}

As a corollary, we obtain the following inequality for mixed volumes.
\begin{corollary}\label{cor:mixed-vol}
    For all $K,L,M\in\mathcal{K}^n$, we have 
    \begin{align*}
        n\left(V_1(K,M)+V_1(L,M)\right)&\geq
\left(\vol_n(K)+\vol_n(L)\right)\left[n+\log(\vol_n(M))\right]\\&-\vol_n(K)\log(\vol_n(K))-\vol_n(L)\log(\vol_n(L))
    \end{align*}
    with equality if and only if $K,L,M$ are translates of one another.
\end{corollary}

\begin{proof}
First note that for any convex body $K\subset\R^n$,
\[
\ent(\mathbbm{1}_K)=-\vol_n(K)\log(\vol_n(K)).
\]
Applying Proposition \ref{thm:entropy} with $f_1=\mathbbm{1}_K,f_2=\mathbbm{1}_L$ and $g=\mathbbm{1}_M$, we get
\begin{align*}
    \delta(\mathbbm{1}_K\sharp\mathbbm{1}_L,\mathbbm{1}_M) &\geq (J(\mathbbm{1}_K)+J(\mathbbm{1}_L))\left[n+\log J(\mathbbm{1}_M)\right]+\ent(\mathbbm{1}_K)+\ent(\mathbbm{1}_L)\\
    &=\left(\vol_n(K)+\vol_n(L)\right)\left[n+\log(\vol_n(M))\right]\\
    &-\vol_n(K)\log(\vol_n(K))-\vol_n(L)\log(\vol_n(L))
\end{align*}
with equality if and only if $\mathbbm{1}_K,\mathbbm{1}_L,\mathbbm{1}_M$ are translates of one another, i.e., if and only if $K,L,M$ are translates of one another.

By the first variation formula, together with the indicator surface area measures $(\mu_{\mathbbm{1}_K},\nu_{\mathbbm{1}_K})=(\vol_n(K)\delta_o,S_K)$, we have
\[
\delta(\mathbbm{1}_K,\mathbbm{1}_M)=\int h_M\,dS_K=nV_1(K,M).
\]
Similarly, $\delta(\mathbbm{1}_L,\mathbbm{1}_M)=nV_1(L,M)$. To conclude the proof, note that by Proposition \ref{prop:Blaschke-operations}(ii) we have 
\[
\delta(\mathbbm{1}_K\sharp\mathbbm{1}_L,\mathbbm{1}_M) =\delta(\mathbbm{1}_K,\mathbbm{1}_M) +\delta(\mathbbm{1}_L,\mathbbm{1}_M)=n\left(V_1(K,M)+V_1(L,M)\right).
\]
The result follows.
\end{proof}

\begin{remark}
    Note that 
    \[
nV_1(K\#L,M)=\int h_M\,dS_{K\# L}=\int h_M\,dS_K+\int h_M\,dS_L=nV_1(K,M)+nV_1(L,M)
    \]
    so the left-hand side of the inequality in Corollary \ref{cor:mixed-vol} can be expressed in terms of the Blaschke sum of bodies. Proposition \ref{prop:Blaschke-operations}(ii) may be regarded as a functional version of this identity.
\end{remark}

The next result states that the entropy is concave with respect to the Blaschke sum.

\begin{theorem}\label{thm:entropy-concave}
Let $f_1,f_2\in\LC_n$ and $0\leq t\leq 1$. Then
\[
\ent\big((1-t)\odot f_1\sharp t\odot f_2\big)\geq (1-t)\ent(f_1)+t\ent(f_2).
\]
If $0<t<1$, then equality holds if and only if $f_1$ and $f_2$ are translates of one another. 
\end{theorem}

\begin{remark}
    In particular, the usual differential entropy $h(f):=-J(f\log f)$ is convex with respect to the Blaschke operations, i.e., if $f_1,f_2\in\LC_n$ are probability distributions, then 
\[
\forall t\in[0,1],\quad h\big((1-t)\odot f_1\sharp t\odot f_2\big)\leq (1-t)h(f_1)+th(f_2).
\]
If $0<t<1$, then equality holds if and only if $f_1,f_2$ are translates.
\end{remark}

\begin{proof}[Proof of Theorem \ref{thm:entropy-concave}]

Set $h:=(1-t)\odot f_1\sharp t\odot f_2$.  Using the identity
\[
\delta(h,h)=J(h)\bigl(n+\log J(h)\bigr)+\ent(h),
\]
which follows from \eqref{eq:delta-f-f}, together with the Blaschke additivity and homogeneity of $\delta$ in the first variable, we get
\[
\delta(h,h)=(1-t)\delta(f_1,h)+t\delta(f_2,h).
\]
Applying Theorem \ref{thm:CF-entropy} to the pairs $(f_1,h)$ and $(f_2,h)$ yields
\begin{align*}
\delta(h,h)
&\ge (1-t)\Big(J(f_1)\bigl(n+\log J(h)\bigr)+\ent(f_1)\Big)  + t\Big(J(f_2)\bigl(n+\log J(h)\bigr)+\ent(f_2)\Big).
\end{align*}
Recall that in Proposition \ref{prop:Blaschke-operations}(viii), we showed that the total mass $J$ is linear under $\sharp$ and $\odot$. Hence $J(h)=(1-t)J(f_1)+tJ(f_2)$. It follows that
\[
\delta(h,h)\geq J(h)\bigl(n+\log J(h)\bigr)+(1-t)\ent(f_1)+t\ent(f_2).
\]
Comparing this with the formula \eqref{eq:delta-f-f} for $\delta(h,h)$ gives
\[
\ent(h)\geq (1-t)\ent(f_1)+t\ent(f_2).
\]

Assume now that $0<t<1$. Equality holds if and only if equality holds in both applications of Theorem \ref{thm:CF-entropy}, that is, if and only if $h$ is a translate of both $f_1$ and $f_2$. Equivalently, $f_1$ and $f_2$ are translates of one another.

Conversely, if $f_1$ and $f_2$ are translates of one another, then they have the same surface area measures. Thus, $h$ has the same surface area measures as $f_1$, so $h$ is a translate of $f_1$ and equality follows.
\end{proof}

It is natural to ask what geometric inequalities we can obtain from this result when choosing indicator functions of convex bodies. As a corollary, we obtain the following Kneser--S\"uss-type inequality.

\begin{corollary}\label{cor:KS-cor-1}
Let $K$ and $L$ be convex bodies in $\mathbb R^n$, and let $0<t<1$. Set
\[
B_t:=(1-t)^{\frac{1}{n-1}}K\# t^{\frac{1}{n-1}}L
\]
and $m_t:=(1-t)\vol_n(K)+t\vol_n(L)$. Then
\begin{equation}\label{eqn:multiplicative-KS-type-ineq}
\vol_n(B_t)^{\frac{n-1}{n}}
\geq m_t\vol_n(K)^{-\frac{(1-t)\vol_n(K)}{n m_t}}\vol_n(L)^{-\frac{t\vol_n(L)}{nm_t}}.
\end{equation}
Equality holds if and only if $K$ and $L$ are translates. 
\end{corollary}

\begin{remark}
    The inequality extends to $t=0$ and $t=1$ by continuity, with equality trivially at the endpoints.
\end{remark}

\begin{remark}
Choosing $t=\frac12$ in Corollary \ref{cor:KS-cor-1} gives
\[
B_{1/2}=\left(\tfrac{1}{2}\right)^{\frac{1}{n-1}}K
\#\left(\tfrac{1}{2}\right)^{\frac{1}{n-1}}L=2^{-\frac{1}{n-1}}(K\# L),
\]
up to translation. Hence, Corollary \ref{cor:KS-cor-1} yields
\[
\vol_n(K\# L)^{\frac{n-1}{n}}
\geq
\big(\vol_n(K)+\vol_n(L)\big)
\vol_n(K)^{-\frac{\vol_n(K)}{n(\vol_n(K)+\vol_n(L))}}\vol_n(L)^{-\frac{\vol_n(L)}{n(\vol_n(K)+\vol_n(L))}}.
\]
Equality holds if and only if $K$ and $L$ are translates.
\end{remark}

\begin{proof}[Proof of Corollary \ref{cor:KS-cor-1}]
    Set $h_t:=(1-t)\odot\mathbbm{1}_K\sharp t\odot\mathbbm{1}_L$. By Proposition \ref{prop:indicators},
    \[
h_t=\lambda_t^{-(n-1)}\mathbbm{1}_{\lambda_t B_t}
    \]
    where 
    \[
    B_t:=(1-t)^{\frac{1}{n-1}}K\# t^{\frac{1}{n-1}}L\quad\text{and}\quad \lambda_t:=\frac{(1-t)\vol_n(K)+t\vol_n(L)}{\vol_n(B_t)}.
    \]
    Also, for any $c>0$ and any convex body $A\in\mathcal{K}^n$,
    \[
\ent(c\mathbbm{1}_A)=(c\log c)\vol_n(A)-c\vol_n(A)\log(c\vol_n(A))=-c\vol_n(A)\log(\vol_n(A)).
    \]
    Thus,
    \[
\ent(h_t)=-m_t\log(\vol(\lambda_t B_t))\quad\text{where}\quad m_t:=(1-t)\vol_n(K)+t\vol_n(L).
    \]
    By the homogeneity of volume, we have \[\vol_n(\lambda_t B_t)=\lambda_t^n\vol_n(B_t)=m_t^n\vol_n(B_t)^{-(n-1)}. \]Hence, Theorem \ref{thm:entropy-concave} becomes
    \[
-m_t\log(m_t^n\vol_n(B_t)^{-(n-1)})\geq -(1-t)\vol_n(K)\log(\vol_n(K))-t\vol_n(L)\log(\vol_n(L)).
    \]
    This is equivalent to \eqref{eqn:multiplicative-KS-type-ineq}. Equality holds if and only if $\mathbbm{1}_K$ and $\mathbbm{1}_L$ are translates, i.e., if and only if $K$ and $L$ are translates.
\end{proof}

\begin{remark}
   Theorem \ref{thm:entropy-concave} may be regarded as a  functional analogue of the classical Kneser--S\"uss inequality for the canonical Blaschke sum. Note that Corollary \ref{cor:KS-cor-1} can also be obtained from the classical Kneser--S\"uss inequality and the weighted AM--GM inequality, so it is  strictly weaker than the classical Kneser--S\"uss inequality. Our proof shows that Corollary \ref{cor:KS-cor-1} can be obtained from the entropy functional rather than the classical  Kneser--S\"uss inequality.
\end{remark}

In the final result of this section, we show that the entropy of a log-concave function can be bounded by the entropy of its mean Blaschke symmetral.

\begin{proposition}\label{prop:entropy-mean-blaschke-symmetral}
    For every $f\in\LC_n$, we have $\ent(f^\sharp)\geq\ent(f)$ with equality if and only if $f$ is a translate of a radial function.
\end{proposition}

\begin{proof}
We first record two elementary invariance properties of the entropy. If $a\in\mathbb R^n$ and
$T_a f(x):=f(x+a)$, then
\[
\ent(T_a f)=\ent(f).
\]
Likewise, if $U\in \mathrm{O}(n)$ and $Uf:=f\circ U$, then
\[
\ent(Uf)=\ent(f).
\]
Both identities follow immediately from the change of variables formula and the fact that translations and orthogonal transformations preserve Lebesgue measure.

We shall also use the following continuity fact. Suppose that $h_k=e^{-\phi_k}\in \LC_n$
hypo-converges to $h=e^{-\phi}\in \LC_n$, and suppose that the sequence is uniformly coercive and uniformly bounded from above in the following sense: there exist constants $a,b>0$ and
$M<\infty$ such that
\[
\phi_k(x)\geq a|x|-b
\qquad\text{for all }x\in\R^n,\ k\in\mathbb N,
\]
and
\[
\max_{\R^n} h_k\leq M
\qquad\text{for all }k\in\mathbb N.
\]
Then
\[
\ent(h_k)\longrightarrow \ent(h)\quad\text{as }k\to\infty.
\]
Indeed, hypo-convergence of $h_k$ to $h$ is equivalent to epi-convergence of $\phi_k$ to $\phi$. By  \cite[Theorem~7.17]{Rockafellar-Wets}, since $\phi_k$ epi-converges to the proper, lower semicontinuous, convex function $\phi$ and $\dom(\phi)$ has nonempty interior, $\phi_k\to\phi$ uniformly on every compact subset of $\R^n$ which does not meet
$\partial\dom(\phi)$. In particular, $\phi_k(x)\to\phi(x)$ for every
\[
x\in \operatorname{int}(\dom(\phi))
\cup\big(\R^n\setminus \overline{\dom(\phi)}\big).
\]
Since $\dom(\phi)$ is convex and has nonempty interior, its boundary has
Lebesgue measure zero. Therefore, $\phi_k(x)\to \phi(x)$ for Lebesgue-a.e. $x\in\R^n$. Equivalently,
\[
h_k(x)=e^{-\phi_k(x)}\longrightarrow e^{-\phi(x)}=h(x)
\quad\text{for Lebesgue-a.e. }x\in\mathbb R^n.
\]
Moreover,
\[
0\leq h_k(x)=e^{-\phi_k(x)}\leq e^b e^{-a|x|},
\]
and the function on the right-hand side is integrable. Therefore, by the dominated convergence theorem,
\[
J(h_k)=\int_{\R^n}h_k(x)\,dx\longrightarrow\int_{\R^n}h(x)\,dx=J(h)\quad\text{as }k\to\infty.
\]
Next, since $\max h_k\leq M$, we have $\phi_k(x)\geq -\log M$ for all $x\in\R^n$ and all $k\in\mathbb{N}$. The scalar function $t\mapsto |t|e^{-t}$ is bounded on $[-\log M,\infty)$ by a constant multiple of $e^{-t/2}$. More precisely, there
exists $C_M<\infty$ such that
\[
|t|e^{-t}\leq C_M e^{-t/2}
\qquad\text{for all }t\geq -\log M.
\]
Thus,
\[
|h_k(x)\log h_k(x)|
=|\phi_k(x)|e^{-\phi_k(x)}
\leq C_M e^{-\phi_k(x)/2}\leq
C_M e^{b/2}e^{-a|x|/2}.
\]
Again, the right-hand side is integrable. Since
\[
h_k(x)\log h_k(x)\longrightarrow h(x)\log h(x)
\]
for a.e. $x$, the dominated convergence theorem gives
\[
J(h_k\log h_k)\longrightarrow
J(h\log h).
\]
Combining this with $J(h_k)\to J(h)$ and the continuity of $s\mapsto s\log s$ on
$(0,\infty)$, we obtain
\[
\ent(h_k)=J(h_k\log h_k)-J(h_k)\log J(h_k)\longrightarrow
J(h\log h)-J(h)\log J(h)
=\ent(h).
\]

We now prove the inequality. Choose the sequence of directions $u_1,u_2,\ldots\in\Sp$ in Theorem~\ref{thm:Blaschke=convergence-thm} so that the subgroup of $\mathrm{O}(n)$ generated by the reflections $R_{u_k}$ is dense in $\mathrm{O}(n)$. Set
\[
f_0:=f,
\qquad
f_k:=B^\sharp_{u_k}f_{k-1}
=B^\sharp_{u_k}\cdots B^\sharp_{u_1}f
\quad(k\geq 1).
\]
By Theorem~\ref{thm:Blaschke=convergence-thm}, there exist translations $v_k\in\mathbb R^n$ such that $\widetilde f_k(x):=f_k(x+v_k)$ 
hypo-converges to $f^\sharp$. For every $h\in\LC_n$ and every $u\in\Sp$, applying Theorem \ref{thm:entropy-concave} with
$f_1=h$, $f_2=R_u h$, and $t=1/2$ we get
\[
\ent(B^\sharp_u h)
=\ent\left(\frac12\odot h\sharp \frac12\odot R_u h\right)
\geq\frac12\ent(h)+\frac12\ent(R_u h)
=\ent(h).
\]
Therefore,
\[
\ent(f_0)\leq \ent(f_1)\leq\cdots\leq
\ent(f_k)\leq\cdots .
\]
Since entropy is translation invariant, we have $\ent(\widetilde f_k)=\ent(f_k)$ for every $k$. The sequence $\widetilde f_k$ is precisely the translated sequence supplied by Theorem~\ref{thm:Blaschke=convergence-thm}.
By the compactness and convergence theorem of Falah and Rotem used in Theorem~\ref{thm:Blaschke=convergence-thm}, the
corresponding sequence has the uniform coercivity and uniform upper bounds required in the continuity argument above. Hence,
\[
\ent(f_k)=\ent(\widetilde f_k)
\longrightarrow
\ent(f^\sharp)\qquad \text{as }k\to\infty.
\]
It follows that
\[
\ent(f^\sharp)\geq \ent(f).
\]

We now prove the equality statement. First, suppose that $f$ is a translate of a radial function. Then its surface area pair is, up to translation, already invariant under the action of $\mathrm{O}(n)$. Hence, the Haar-averaged surface area pair defining $f^\sharp$ is the same pair. By the uniqueness part of the functional Minkowski problem, $f^\sharp$ is a translate of $f$. Therefore, $\ent(f^\sharp)=\ent(f)$.

Conversely, suppose that $\ent(f^\sharp)=\ent(f)$. Since the sequence $\ent(f_k)$ is monotone nondecreasing and converges to
$\ent(f^\sharp)$, we must have $\ent(f_k)=\ent(f_{k-1})$
 for every $k\geq 1$. For each $k$, equality in the application of Theorem~\ref{thm:entropy-concave} to $f_{k-1}$ and $R_{u_k}f_{k-1}$ therefore implies that
$f_{k-1}$ and $R_{u_k}f_{k-1}$ are translates of one another.

We claim that this implies that $f_k$ is a translate of $f_{k-1}$. Indeed, if
$R_{u_k}f_{k-1}$ is a translate of $f_{k-1}$, then $f_{k-1}$ and $R_{u_k}f_{k-1}$ have the same
surface area measures. Hence
\[
\left(\mu_{f_k},\nu_{f_k}\right)
=\frac12\left(\mu_{f_{k-1}},\nu_{f_{k-1}}\right)
+\frac12\left(\mu_{R_{u_k}f_{k-1}},\nu_{R_{u_k}f_{k-1}}\right)
=\left(\mu_{f_{k-1}},\nu_{f_{k-1}}\right).
\]
By uniqueness in the functional Minkowski problem, $f_k$ is a translate of $f_{k-1}$. Inductively, every $f_k$ is a translate of the original function $f$.

Since $f_{k-1}$ is a translate of $f$, and since $f_{k-1}$ is a translate of
$R_{u_k}f_{k-1}$, it follows that there exists an affine isometry $A_k$ of $\mathbb R^n$ whose
linear part is $R_{u_k}$ and such that
\[
f(A_k x)=f(x)
\qquad\text{for every }x\in\mathbb R^n.
\]
Indeed, this is just the conjugate, by a translation, of the relation saying that
$f_{k-1}$ and $R_{u_k}f_{k-1}$ are translates.

Choose $s>0$ such that the superlevel set
$K:=\{x\in\mathbb R^n:f(x)\geq s\}$ 
is a convex body. Since $f\in \LC_n$ has positive mass, such an $s$ exists. Each affine isometry $A_k$ preserves $f$, hence preserves $K$. Let $c$ be the centroid of $K$. Because affine
isometries preserving $K$ preserve its centroid, we have $A_k c=c$ for every $k$. Since the linear part of $A_k$ is $R_{u_k}$ and $A_k$ fixes $c$, we also have $A_k x=c+R_{u_k}(x-c)$ for every $x\in\R^n$. Thus,
\[
f(c+R_{u_k}x)=f(c+x)
\qquad\text{for all }x\in\R^n,\ k\geq 1.
\]
Set $g(x):=f(c+x)$. Then
\[
g(R_{u_k}x)=g(x)
\qquad\text{for every }x\in\R^n,\ k\geq 1.
\]
Hence, $g$ is invariant under the subgroup of $\mathrm{O}(n)$ generated by the reflections $R_{u_k}$. By construction, this subgroup is dense in $\mathrm{O}(n)$. We now show that this implies $g$ is invariant under all of $\mathrm{O}(n)$. Let $U\in \mathrm{O}(n)$ and choose a sequence $\{U_j\}$ from the generated subgroup with $U_j\to U$. Since $g(U_j x)=g(x)$ for all $j$, by the upper
semicontinuity of $g$ we obtain
\[
g(Ux)\geq \limsup_{j\to\infty}g(U_j x)=g(x).
\]
Applying the same argument to $U^{-1}$ gives the reverse inequality. Therefore,
\[
g(Ux)=g(x)
\qquad\text{for every }U\in \mathrm{O}(n),\ x\in\R^n.
\]
Thus, $g$ is radial. Since $g(x)=f(c+x)$, the original function $f$ is a translate of a radial function. This proves the equality characterization and completes the proof.
\end{proof}

\section{Functional affine surface area and affine isoperimetric inequalities}\label{sec:affine-surface-area}

Let $\mathcal{S}_c^n$ denote the set of all star bodies in $\R^n$ whose centroid is at the origin. Lutwak \cite{Lutwak-1991} defined the \emph{affine surface area} $\Omega(K)$ of $K\in\mathcal{K}^n$ by
\begin{equation}\label{eq:geo-as}
\Omega(K)=n^{\frac{1}{n+1}}\left(\inf_{Q\in\mathcal{S}_c^n}\left\{nV_1(K,Q^\circ)\vol_n(Q)^{\frac{1}{n}}\right\}\right)^{\frac{n}{n+1}}.
\end{equation}
\noindent Note that if $Q\in\mathcal{S}_c^n$, then $Q^\circ\in\mathcal{K}^n_o$. In terms of the surface area measure of $K$, this can also be expressed as
\begin{equation}\label{eq:geo-as-cvx}
\Omega(K)=n^{\frac{1}{n+1}}\left(\inf_{Q\in\mathcal{S}_c^n}\left\{n\vol_n(Q)^{\frac{1}{n}}\int_{\Sp}h_{Q^\circ}(u)\,dS_K(u)\right\}\right)^{\frac{n}{n+1}}.
\end{equation} 
Now observe that 
\[
\delta(\mathbbm{1}_K,\mathbbm{1}_{Q^\circ})=\int_{\Sp}h_{Q^\circ}\,dS_K=nV_1(K,Q^\circ).
\]

This motivates the following notion of functional affine surface area:

\begin{definition}[Functional affine surface area]
Let $f\in\LC_n$. The \emph{affine surface area of $f$} is
\[
\Omega_\sharp(f):=n^{\frac{1}{n+1}}\left(\inf_{Q\in\mathcal{S}_c^n}\left\{\delta(f,\mathbbm{1}_{Q^\circ})\vol_n(Q)^{\frac{1}{n}}\right\}\right)^{\frac{n}{n+1}}.
\]
Also, for fixed $Q\in\mathcal{S}_c^n$, we set
\[
\Omega_Q(f):=\delta(f,\mathbbm{1}_{Q^\circ})\vol_n(Q)^{\frac{1}{n}}
\]
so that  $\Omega_\sharp(f):=n^{\frac{1}{n+1}}(\inf_{Q\in\mathcal{S}_c^n}\Omega_Q(f))^{\frac{n}{n+1}}$. 
\end{definition}
Since $Q\in\mathcal{S}_c^n$, its polar $Q^\circ$ is convex, so $\mathbbm{1}_{Q^\circ}$ is a log-concave function. Hence $\Omega_\sharp(f)$ is well-defined.

\vspace{2mm}

Next, let us highlight some important properties of this functional affine surface area. 

\begin{proposition}\label{prop:affine-sa}
    Let $f\in\LC_n$. 
    \begin{itemize}
        \item[(i)] \rm{(Affine covariance of the first variation)}  If $T\in{\rm GL}_n(\R)$, then $\delta(f\circ T,g)=|\det T|^{-1}\delta(f,g\circ T^{-1})$. In particular, if $T\in{\rm SL}_n(\R)$, then $\delta(f\circ T,g)=\delta(f,g\circ T^{-1})$.
        
        \item[(ii)] {(Affine covariance of $\Omega_\sharp$)} If $T\in{\rm GL}_n(\R)$, then $\Omega_\sharp(f\circ T)=|\det T|^{-\frac{n-1}{n+1}}\Omega_\sharp(f)$. In particular, if $T\in{\rm SL}_n(\R)$ then $\Omega_\sharp(f\circ T)=\Omega_\sharp(f)$.

        \item[(iii)] \rm{(Upper semicontinuity)} 
The functional $\Omega_\sharp$ is upper semicontinuous with respect to cosmic convergence of surface area measures. More precisely, suppose that $f_k,f\in \LC_n$ and $(\mu_{f_k},\nu_{f_k})\to (\mu_f,\nu_f)$ cosmically as $k\to\infty$. Then, for every fixed $Q\in\mathcal{S}_c^n$, we have $\Omega_Q(f_k)\to \Omega_Q(f)$. Consequently,
\[
\limsup_{k\to\infty}\Omega_\sharp(f_k)\leq \Omega_\sharp(f).
\] 

        \item[(iv)] \rm{(Homogeneity)} If $\lambda>0$, then $\Omega_\sharp (\lambda f)=\lambda^{\frac{n}{n+1}}\Omega_\sharp(f)$.
    \end{itemize}

    \begin{proof}
(i) Let $f=e^{-\varphi}, g=e^{-\psi}\in\LC_n$ and set $K_f:=\supp(f)$. Throughout the proof, the normal and area element transformation formulas used hold for convex bodies at $\mathcal{H}^{n-1}$-a.e. boundary points (see, e.g., \cite[Theorem 25.5]{RockafellarBook}). The first variation can be expressed as
\begin{equation}\label{eq:first-variation-2}
\delta(f,g)=\int_{\R^n}h_g(\nabla\varphi(x))f(x)\,dx+\int_{\partial K_f}h_{\supp(g)}(n_K(y))f(y)\,d\mathcal{H}^{n-1}(y).
\end{equation}
For $T\in\mathrm{GL}_n$, set $\widetilde{f}:=f\circ T$ and $\widetilde{\varphi}:=\varphi\circ T$, and note that $\widetilde{f}=e^{-\widetilde{\varphi}}$. We need to compute $\delta(\widetilde{f},g)$.

We first consider the interior term (first integral) in \eqref{eq:first-variation-2}. By the chain rule, for every $x\in\R^n$ we have $\nabla\widetilde{\varphi}(x)=T^\top\nabla\varphi(Tx)$. Hence, substituting $y=Tx$ we derive that
\begin{align*}
    \int_{\R^n}h_g(\nabla\widetilde{\varphi}(x))\widetilde{f}(x)\,dx=\int_{\R^n}h_g(T^\top\nabla\varphi(Tx))f(Tx)\,dx=|\det(T)|^{-1}\int_{\R^n}h_g(T^\top\nabla\varphi(y))f(y)\,dy.
\end{align*}
Now $g\circ T^{-1}=e^{-(\psi\circ T^{-1})}$, so its support function is $h_{g\circ T^{-1}}=\mathcal{L}(\psi\circ T^{-1})$. For every $z\in\R^n$,
\[
\mathcal{L}(\psi\circ T^{-1})(z)=\sup_{x\in\R^n}\left(\langle z,x\rangle-\psi(T^{-1}x)\right).
\]
Substituting $x=Ty$, this becomes
\[
\mathcal{L}(\psi\circ T^{-1})(z)=\sup_{y\in\R^n}\left(\langle z,Ty\rangle-\psi(y)\right)=\sup_{y\in\R^n}\left(\langle T^\top z,y\rangle-\psi(y)\right)=(\mathcal{L}\psi)(T^\top z)=h_g(T^\top z).
\]
Therefore, $h_{g\circ T^{-1}}(z)=h_g(T^\top z)$, and hence the first integral becomes
\begin{equation}\label{eq:interior-eq-1}
    \int_{\R^n}h_g(\nabla\widetilde{\varphi}(x))\widetilde{f}(x)\,dx=|\det(T)|^{-1}\int_{\R^n}h_{g\circ T^{-1}}(\nabla\varphi(y))f(y)\,dy.
\end{equation}

For the boundary term in \eqref{eq:first-variation-2} (i.e., the second integral), set $\widetilde{K}_f:=\supp(\widetilde{f})$. For $y\in K_f$, the corresponding point on $\partial \widetilde{K}_f$ is $x=T^{-1}y$. Using the formulas
\[
n_{\widetilde{K}_f}(T^{-1}y)=\frac{T^\top n_K(y)}{|T^\top n_K(y)|}\quad\text{and}\quad d\mathcal{H}^{n-1}(T^{-1}y)=|\det(T)|^{-1}|T^\top n_K(y)|\,d\mathcal{H}^{n-1}(y),
\]
we get
\begin{align*}
    &\int_{\partial\widetilde{K}_f}h_{\supp(g)}(n_{\widetilde{K}_f}(x))\widetilde{f}(x)\,d\mathcal{H}^{n-1}(x) \\&=\int_{\partial K}h_{\supp(g)}\left(\frac{T^\top n_K(y)}{|T^\top n_K(y)|}\right)f(y)|\det(T)|^{-1}|T^\top n_K(y)|\,d\mathcal{H}^{n-1}(y)\\
    &=|\det(T)|^{-1}\int_{\partial K}h_{\supp(g)}(T^\top n_K(y))f(y)\,d\mathcal{H}^{n-1}(y).
\end{align*}
In the last line, we used the fact that the support function $h_{\supp(g)}$ is 1-homogeneous.  Now $\supp(g\circ T^{-1})=T(\supp(g))$, so for every $v\in\R^n$,
\[
h_{\supp(g\circ T^{-1})}(v)=h_{T(\supp(g))}(v)=h_{\supp(g)}(T^\top v).
\]
Thus, the previous integral becomes
\begin{equation}\label{eq:boundary-term-2}
\int_{\partial\widetilde{K}_f}h_{\supp(g)}(n_{\widetilde{K}_f}(x))\widetilde{f}(x)\,d\mathcal{H}^{n-1}(x) =|\det(T)|^{-1}\int_{\partial K}h_{\supp(g\circ T^{-1})}(n_K(y))f(y)\,d\mathcal{H}^{n-1}(y).
\end{equation}

Therefore, by \eqref{eq:interior-eq-1} and \eqref{eq:boundary-term-2}, we finally obtain
\begin{align*}
    &\delta(f\circ T,g)=\delta(\widetilde{f},g)\\
    &=|\det(T)|^{-1}\biggl[\int_{\R^n}h_{g\circ T^{-1}}(\nabla\varphi(y))f(y)\,dy+\int_{\partial K}h_{\supp(g\circ T^{-1})}(n_K(y))f(y)\,d\mathcal{H}^{n-1}(y)\biggr]\\
    &=|\det(T)|^{-1}\delta(f,g\circ T^{-1}).
\end{align*}

(ii) By the definition of $\Omega_\sharp$ and part (i),
\begin{align*}
    \Omega_\sharp(f\circ T)&=n^{\frac{1}{n+1}}\left(\inf_{Q\in\mathcal{S}_c^n}\left\{\delta(f\circ T,\mathbbm{1}_{Q^\circ})\vol_n(Q)^{\frac{1}{n}}\right\}\right)^{\frac{n}{n+1}}\\
    &=|\det(T)|^{-\frac{n}{n+1}}n^{\frac{1}{n+1}}\left(\inf_{Q\in\mathcal{S}_c^n}\left\{\delta(f,\mathbbm{1}_{Q^\circ}\circ T^{-1})\vol_n(Q)^{\frac{1}{n}}\right\}\right)^{\frac{n}{n+1}}.
\end{align*}
Now $\mathbbm{1}_{Q^\circ}\circ T^{-1}=\mathbbm{1}_{TQ^\circ}$ and $(T^{-\top}Q)^\circ=TQ^\circ$, so $\mathbbm{1}_{Q^\circ}\circ T^{-1}=\mathbbm{1}_{(T^{-\top}Q)^\circ}$. Set $P:=T^{-\top}Q$. Since $Q\in\mathcal{S}_c^n$ and $T\in\mathrm{GL}_n$, we have $P\in\mathcal{S}_c^n$. Moreover, $Q=T^\top P$ and $\vol_n(Q)=|\det(T)|\vol_n(P)$. Therefore,
\begin{align*}
    \Omega_\sharp(f\circ T)&=|\det(T)|^{-\frac{n}{n+1}}n^{\frac{1}{n+1}}\left(\inf_{P\in\mathcal{S}_c^n}\left\{\delta(f,\mathbbm{1}_{P^\circ})(|\det(T)|\vol_n(P))^{\frac{1}{n}}\right\}\right)^{\frac{n}{n+1}}\\
    &=|\det(T)|^{-\frac{n-1}{n+1}}n^{\frac{1}{n+1}}\left(\inf_{P\in\mathcal{S}_c^n}\left\{\delta(f,\mathbbm{1}_{P^\circ})\vol_n(P)^{\frac{1}{n}}\right\}\right)^{\frac{n}{n+1}}=|\det(T)|^{-\frac{n-1}{n+1}}\Omega_\sharp(f).
\end{align*}

(iii) Fix $Q\in\mathcal S_c^n$. By definition, $\Omega_Q(f)=\delta(f,\mathbbm{1}_{Q^\circ})\vol_n(Q)^{1/n}$. Since $Q\in\mathcal{S}_c^n$, its polar $Q^\circ$ is a convex body containing the origin in its
interior. Hence $\mathbbm{1}_{Q^\circ}\in \LC_n$, $h_{\mathbbm{1}_{Q^\circ}}=h_{Q^\circ}$, and $h_{\operatorname{supp}(\mathbbm{1}_{Q^\circ})}=h_{Q^\circ}$. 
Therefore, by the first variation formula,
\[
\delta(f_k,\mathbbm{1}_{Q^\circ})
=
\int_{\mathbb R^n} h_{Q^\circ}(x)\,d\mu_{f_k}(x)
+
\int_{\mathbb S^{n-1}} h_{Q^\circ}(\theta)\,d\nu_{f_k}(\theta).
\]
Likewise,
\[
\delta(f,\mathbbm{1}_{Q^\circ})
=
\int_{\mathbb R^n} h_{Q^\circ}(x)\,d\mu_f(x)
+
\int_{\mathbb S^{n-1}} h_{Q^\circ}(\theta)\,d\nu_f(\theta).
\]

The function $h_{Q^\circ}$ is continuous and positively $1$-homogeneous. Hence, it is an admissible test function for cosmic convergence. Since $(\mu_{f_k},\nu_{f_k})\to(\mu_f,\nu_f)$ cosmically, we obtain
\[
\int_{\R^n} h_{Q^\circ}(x)\,d\mu_{f_k}(x)
+\int_{\Sp} h_{Q^\circ}(\theta)\,d\nu_{f_k}(\theta)
\longrightarrow
\int_{\R^n} h_{Q^\circ}(x)\,d\mu_f(x)
+\int_{\Sp} h_{Q^\circ}(\theta)\,d\nu_f(\theta).
\]
Thus, $\delta(f_k,\mathbbm{1}_{Q^\circ})\to \delta(f,\mathbbm{1}_{Q^\circ})$. 
Multiplying by the fixed constant $\vol_n(Q)^{1/n}$, we get
\[
\Omega_Q(f_k)\longrightarrow \Omega_Q(f).
\]

It remains to pass from the pointwise convergence of $\Omega_Q$ to the upper semicontinuity
of $\Omega_\sharp$. Set
\[
I(f):=\inf_{Q\in\mathcal S_c^n}\Omega_Q(f).
\]
Then
\[
\Omega_\sharp(f)=n^{1/(n+1)}I(f)^{n/(n+1)}.
\]
Let $\varepsilon>0$. By the definition of the infimum, there exists 
$Q_\varepsilon\in\mathcal{S}_c^n$ such that
\[
\Omega_{Q_\varepsilon}(f)\leq I(f)+\varepsilon.
\]
For every $k\in\mathbb{N}$,
\[
I(f_k)=\inf_{Q\in\mathcal S_c^n}\Omega_Q(f_k)
\leq
\Omega_{Q_\varepsilon}(f_k).
\]
Taking the limit superior and using the convergence already proved for the fixed body
$Q_\varepsilon$, we get
\[
\limsup_{k\to\infty} I(f_k)
\leq
\lim_{k\to\infty}\Omega_{Q_\varepsilon}(f_k)=\Omega_{Q_\varepsilon}(f)
\leq I(f)+\varepsilon.
\]
Since $\varepsilon>0$ was arbitrary,
\[
\limsup_{k\to\infty} I(f_k)\leq I(f).
\]
Finally, the function $s\mapsto n^{1/(n+1)}s^{n/(n+1)}$ is increasing and continuous on
$[0,\infty)$, so
\[
\limsup_{k\to\infty}\Omega_\sharp(f_k)
=
\limsup_{k\to\infty}
n^{1/(n+1)}I(f_k)^{n/(n+1)}
\leq
n^{1/(n+1)}I(f)^{n/(n+1)}
=\Omega_\sharp(f).
\]
This proves the upper semicontinuity in (iii).

(iv) Let $\lambda>0$. By the definition of $\Omega_\sharp$ and the property $\delta(\lambda f,g)=\lambda\delta(f,g)$ (which follows from Theorem \ref{thm:rotem-falah-variation}),
\begin{align*}
    \Omega_{\sharp}(\lambda f)&=n^{\frac{1}{n+1}}\left(\inf_{Q\in\mathcal{S}_c^n}\left\{\delta(\lambda f,\mathbbm{1}_{Q^\circ})\vol_n(Q)^{\frac{1}{n}}\right\}\right)^{\frac{n}{n+1}}
    =\lambda^{\frac{n}{n+1}}\Omega_\sharp(f).
\end{align*}
    \end{proof}
\end{proposition}

We are now ready to prove the main result of this section, which is the following concavity property of the functional affine surface area. It states that $\Omega_\sharp$ is $\frac{n+1}{n}$-concave with respect to the functional Blaschke addition.

\begin{theorem}\label{thm:concavity-of-functional-affine-surface-area}
    For all $f_1,f_2\in\LC_n$ and all $t\in(0,1)$, we have
    \[
\Omega_\sharp\left((1-t)\odot f_1\sharp t\odot f_2\right)^{\frac{n+1}{n}} \geq (1-t)\Omega_\sharp(f_1)^{\frac{n+1}{n}}+t\Omega_\sharp(f_2)^{\frac{n+1}{n}}.
    \]
\end{theorem}

\begin{remark}
    Let $K,L\in\mathcal{K}_o^n$ and choose $f_1=\mathbbm{1}_K,f_2=\mathbbm{1}_L$in Theorem \ref{thm:concavity-of-functional-affine-surface-area}. Invoking Proposition \ref{prop:indicators} and the homogeneity of $\Omega_\sharp$ under vertical scaling, from Theorem \ref{thm:concavity-of-functional-affine-surface-area} we recover Lutwak's result \cite[Proposition 4.5]{Lutwak-1991} for convex bodies, which states that for any $K,L\in\mathcal{K}^n$,
    \[
    \Omega(K\# L)^{\frac{n+1}{n}}\geq \Omega(K)^{\frac{n+1}{n}}+\Omega(L)^{\frac{n+1}{n}}.
    \]
\end{remark}

\begin{proof}[Proof of Theorem \ref{thm:concavity-of-functional-affine-surface-area}]
    Fix $Q\in\mathcal{S}_c^n$, and for $f\in\LC_n$, recall that  $\Omega_Q(f):=\delta(f,\mathbbm{1}_{Q^\circ})\vol_n(Q)^{\frac{1}{n}}$. By (ii) and (iii) of Proposition \ref{prop:Blaschke-operations}, the functional $\Omega_Q$ is $\sharp$-linear and $\odot$-homogeneous of degree 1, i.e., for all $f_1,f_2\in\LC_n$ and all $\lambda\geq 0$,
    \[
\Omega_Q(f_1\sharp f_2)=\Omega_Q(f_1)+\Omega_Q(f_2),\qquad \Omega_Q(\lambda\odot f)=\lambda \Omega_Q(f).
    \]
    Set $\widetilde{\Omega}_\sharp(f):=\Omega_\sharp(f)^{\frac{n+1}{n}}$, so that $\widetilde{\Omega}_\sharp(f)=n^{1/n}\inf_{Q\in\mathcal{S}_c^n}\Omega_Q(f)$. Also, for each $t\in(0,1)$ we set $f_t:=(1-t)\odot f_1\sharp t\odot f_2$. We will now compute $\widetilde{\Omega}_\sharp(f_t)$. By the linearity and homogeneity of $\delta(\cdot,g)$ with respect to the Blaschke operations,
    \begin{align*}
\Omega_Q(f_t)&=\delta((1-t)\odot f_1\sharp t\odot f_2,\mathbbm{1}_{Q^\circ})\vol_n(Q)^{\frac{1}{n}}\\\
&=\left[(1-t)\delta(f_1,\mathbbm{1}_{Q^\circ})+t\delta(f_2,\mathbbm{1}_{Q^\circ})\right]\vol_n(Q)^{\frac{1}{n}}\\
&=(1-t)\Omega_Q(f_1)+t\Omega_Q(f_2).
    \end{align*}
    Thus,
    \[
\widetilde{\Omega}_\sharp(f_t)=n^{\frac{1}{n}}\inf_{Q\in\mathcal{S}_c^n}\left[(1-t)\Omega_Q(f_1)+t\Omega_Q(f_2)\right].
    \]
    Now for every fixed $Q\in\mathcal{S}_c^n$, 
    \[
\Omega_Q(f_i) \geq \inf_{P\in\mathcal{S}_c^n}\Omega_P(f_i),\quad i=1,2,
    \]
    so
    \[
(1-t)\Omega_Q(f_1)+t\Omega_Q(f_2)\geq (1-t)\inf_{P\in\mathcal{S}_c^n}\Omega_P(f_1)+t\inf_{P'\in\mathcal{S}_c^n}\Omega_{P'}(f_2).
    \]
    Since this lower bound is valid for every $Q\in\mathcal{S}_c^n$, taking the infimum over all such $Q$ yields
    \begin{align*}
\widetilde{\Omega}_\sharp(f_t)&=n^{\frac{1}{n}}\inf_{Q\in\mathcal{S}_c^n}\left[(1-t)\Omega_Q(f_1)+t\Omega_Q(f_2)\right]\\
&\geq (1-t)n^{\frac{1}{n}}\inf_{P\in\mathcal{S}_c^n}\Omega_P(f_1)+tn^{\frac{1}{n}}\inf_{P'\in\mathcal{S}_c^n}\Omega_{P'}(f_2)=(1-t)\widetilde{\Omega}_\sharp(f_1)+t\widetilde{\Omega}_\sharp(f_2).
    \end{align*}
    Finally, substituting $\widetilde{\Omega}_\sharp(f)=\Omega_\sharp(f)^{\frac{n+1}{n}}$, we get
    \[
\Omega_\sharp\left((1-t)\odot f_1\sharp t\odot f_2\right)^{\frac{n+1}{n}} \geq (1-t)\Omega_\sharp(f_1)^{\frac{n+1}{n}}+t\Omega_\sharp(f_2)^{\frac{n+1}{n}}.
    \]
\end{proof}

As an immediate corollary, we derive that the  affine surface area is monotone increasing with respect to Blaschke symmetrization:
\begin{corollary}\label{cor:blaschke-monotone-as-1}
    Let $f\in\LC_n$ and $u\in\Sp$. Then
    \[
\Omega_\sharp(B_u^\sharp f)\geq \Omega_\sharp(f).
    \]
   Equality holds if $f$ is invariant under reflections about $u^\perp$, i.e., $f=R_u f$.
\end{corollary}

\begin{proof}
    The desired inequality follows the definition of $B_u^\sharp f$ and choosing $f_1=f, f_2=R_u f\in\LC_n$ and $t=1/2$ in Theorem \ref{thm:concavity-of-functional-affine-surface-area}, and we also use $\Omega_\sharp(R_u f)=\Omega_\sharp(f)$. If $f=R_u f$, then by Proposition \ref{prop-Blaschke-symmetral}(iii) we have $B_u^\sharp f=f$. Therefore, $\Omega_\sharp(B_u^\sharp f)=\Omega_\sharp(f)$.
\end{proof}

We now have all of the ingredients to prove the following affine isoperimetric inequality (comparison theorem):
\begin{theorem}\label{thm:blaschke-monotone-as}
For every $f\in\LC_n$,
\[
\Omega_\sharp(f)\leq\Omega_\sharp(f^\sharp).
\]
Equality holds if $f$ is radial.
\end{theorem}

\begin{proof}
 By Theorem \ref{thm:Blaschke=convergence-thm}, there exists a sequence of iterated Blaschke symmetrizations $\tilde{f}_k(x)=f_k(x+v_k)$, where $f_k=B_{u_k}^\sharp\cdots B_{u_1}^\sharp f$ and with corresponding translations  $v_k\in\R^n$, that hypo-converges to the mean Blaschke symmetral $f^\sharp\in\LC_n$. Since translations do not change the surface area measures, we have
\[
\Omega_\sharp(\widetilde f_k)=\Omega_\sharp(f_k)
\]
for every $k$. Moreover, by \cite[Theorem 1.10]{Falah-Rotem}, the hypo-convergence $\widetilde f_k\to f^\sharp$ implies the cosmic convergence
\[
(\mu_{\widetilde f_k},\nu_{\widetilde f_k})
\longrightarrow(\mu_{f^\sharp},\nu_{f^\sharp}).
\]
Hence, by Proposition \ref{prop:affine-sa}(iii) we get
\[
\Omega_\sharp(f^\sharp)\geq \limsup_{k\to\infty}\Omega_\sharp(\widetilde f_k)
=\limsup_{k\to\infty}\Omega_\sharp(f_k).
\]
By Corollary \ref{cor:blaschke-monotone-as-1}, the sequence $\Omega_\sharp(f_k)$ is nondecreasing, and therefore
\[
\Omega_\sharp(f^\sharp)\geq\limsup_{k\to\infty}\Omega_\sharp(\tilde{f}_k)\geq\limsup_{k\to\infty}\Omega_\sharp(f)=\Omega_\sharp(f).
    \]
    Equality holds in each inequality if $f$ is radial.
\end{proof}

Next we compute the affine surface area of the mean Blaschke symmetral. 
\begin{lemma}\label{lem:as-hypo-symmetral}
    For every $f\in\LC_n$, we have
    \begin{equation}
        \Omega_\sharp(f^\sharp) = (n\omega_n)^{\frac{1}{n+1}}W_1(f)^{\frac{n}{n+1}}.
    \end{equation}
\end{lemma}

The proof of Lemma \ref{lem:as-hypo-symmetral} can be found in the appendix. We thus obtain the following affine isoperimetric inequality.
\begin{corollary}\label{cor:isoperimetric-ineq-as}
    For every $f\in\LC_n$, we have
    \[
    \Omega_\sharp(f)\leq (n\omega_n)^{\frac{1}{n+1}}W_1(f)^{\frac{n}{n+1}}.
    \]
Equality holds if $f$ is radial.
\end{corollary}

\begin{proof}
    Since $f^\sharp$ is radial and $W_1(f)=W_1(f^\sharp)$, we have
    \[
\Omega_\sharp(f^\sharp)= (n\omega_n)^{\frac{1}{n+1}}W_1(f^\sharp)^{\frac{n}{n+1}}=(n\omega_n)^{\frac{1}{n+1}}W_1(f)^{\frac{n}{n+1}}.
    \]
    The desired inequality now follows from  Theorem \ref{thm:blaschke-monotone-as} and Lemma \ref{lem:as-hypo-symmetral}. If $f$ is radial, then by Theorem \ref{thm:blaschke-monotone-as} we have $\Omega_\sharp(f)=\Omega_\sharp(f^\sharp)$ and thus $\Omega_\sharp(f)=(n\omega_n)^{\frac{1}{n+1}}W_1(f)^{\frac{n}{n+1}}$.
\end{proof}

Given a convex body $K\in\mathcal{K}^n$, the \emph{geominimal surface area} $G(K)$ of $K$ is
\[
G(K) = \inf_{Q\in\mathcal{K}_c^n}\left\{nV_1(K,Q^\circ)\left(\frac{\vol_n(Q)}{\vol_n(B_2^n)}\right)^{\frac{1}{n}}\right\}.
\]
This concept was introduced by Petty \cite{Petty-1974} (see also Lutwak \cite{Lutwak-1991}). This leads to the following functional version, see \cite{Fang-Yang}.

\begin{definition}\label{def:functional-geominimal-surface-area}
    Let $f\in\LC_n$. The \emph{geominimal surface area} $G_\sharp(f)$ of $f$ is
    \[
G_\sharp(f):=\inf_{Q\in\mathcal{K}_c^n}\left\{\delta(f,\mathbbm{1}_{Q^\circ})\left(\frac{\vol_n(Q)}{\vol_n(B_2^n)}\right)^{\frac{1}{n}}\right\}.
    \]
\end{definition}
This version of functional geominimal surface area $G_\sharp$ can be obtained from the one in \cite{Fang-Yang} by restricting their variational problem to test functions of the form $\mathbbm{1}_{Q^\circ}$, $Q\in\mathcal{K}_c^n$, and differs only by the normalization factor $\omega_n^{-1/n}$. Note that for a convex body $K\in\mathcal{K}^n$ and $f=\mathbbm{1}_K$, we recover the classical geominimal surface area, i.e.,  $G_\sharp(\mathbbm{1}_K)=G(K)$.

Petty \cite{Petty-1974}  proved that for any convex body $K\in\mathcal{K}^n$,
\[
\Omega(K)^{n+1}\leq n\omega_n G(K)^n.
\]
Lutwak later gave a new proof of this result in \cite[Proposition (6.4)]{Lutwak-1991}. We prove a functional analogue in the following
\begin{proposition}\label{prop:geo-as-ineq}
    For all $f\in\LC_n$,
    \[
\Omega_\sharp(f)^{n+1} \leq n\omega_n G_\sharp(f)^n.
    \]
\end{proposition}
Choosing the indicator function $f=\mathbbm{1}_K$ where $K\in\mathcal{K}^n$, we recover \cite[Proposition (6.4)]{Lutwak-1991}.

\begin{proof}
    Since $\mathcal{K}_c^n\subset\mathcal{S}_c^n$, by the definitions of $G_\sharp(f)$ and $\Omega_\sharp(f)$, for all $f\in\LC_n$ we have
    \begin{align*}
        G_\sharp(f) &=\inf_{Q\in\mathcal{K}_c^n}\left\{\delta(f,\mathbbm{1}_{Q^\circ})\left(\frac{\vol_n(Q)}{\vol_n(B_2^n)}\right)^{\frac{1}{n}}\right\}\\
        &\geq \inf_{Q\in\mathcal{S}_c^n}\left\{\delta(f,\mathbbm{1}_{Q^\circ})\left(\frac{\vol_n(Q)}{\vol_n(B_2^n)}\right)^{\frac{1}{n}}\right\}
        =(n\omega_n)^{-\frac{1}{n}}\Omega_\sharp(f)^{\frac{n+1}{n}}.
    \end{align*}
    Rearranging terms, we obtain the desired inequality.
\end{proof}

\subsection{A general comparison principle}

We conclude this section with a general comparison principle for a log-concave function and its mean Blaschke symmetral. 

\begin{theorem}\label{thm:meta-theorem}
     Let $F:\LC_n\to[0,\infty)$.
     \begin{itemize}
\item[(i)] Suppose that $F$ is upper semicontinuous with respect to the topology of hypo-convergence on $\LC_n$, invariant under reflections and translations, and increasing with respect to  Blaschke  symmetrizations, i.e., $F(g)\leq F(B_u^\sharp g)$ for all $g\in\LC_n$ and all $u\in\Sp$. Then for all $f\in\LC_n$, we have $F(f) \leq F(f^\sharp)$.

\item[(ii)] Suppose that $F$ is lower semicontinuous  with respect to the topology of hypo-convergence on $\LC_n$, invariant under reflections and translations, and decreasing with respect to  Blaschke  symmetrizations, i.e., $F(g)\geq F(B_u^\sharp g)$ for all  $g\in\LC_n$ and all $u\in\Sp$. Then for all $f\in\LC_n$, we have $F(f) \geq F(f^\sharp)$.
\end{itemize}
\end{theorem}
In particular,  Theorem \ref{thm:blaschke-monotone-as} is an instance of this principle, and Proposition \ref{prop:entropy-mean-blaschke-symmetral} follows from the same scheme
after establishing the required entropy continuity along the symmetrizing sequence.

\begin{proof}[Proof of Theorem \ref{thm:meta-theorem}]
We only prove (i) as the proof of (ii) is analogous with the inequalities reversed. Let $f_k=B^\sharp_{u_k}\cdots B^\sharp_{u_1}f$ be the sequence from Theorem~\ref{thm:Blaschke=convergence-thm}, and let $\widetilde{f}_k(x)=f_k(x+v_k)$ be the translated representatives hypo-converging to $f^\sharp$. Since $F$ is increasing under
Blaschke symmetrizations,
\[
F(f)\leq F(f_1)\leq F(f_2)\leq\cdots\leq F(f_k).
\]
Since $F$ is translation invariant, we have $F(f_k)=F(\widetilde f_k)$ for every $k$. Thus by the upper semicontinuity of $F$,
\[
F(f^\sharp)\ge \limsup_{k\to\infty}F(\widetilde f_k)
=\limsup_{k\to\infty}F(f_k)\geq F(f).
\]
This proves (i). The proof of (ii) is the same, using instead the  lower semicontinuity and the monotonicity assumption $F(g)\geq F(B^\sharp_u g)$.
\end{proof}

\section{Appendix}

\subsection{Proof of Lemma \ref{thm:SA-layer-cake}}
For $s>0$, set $K_s:=\lev_{\geq s} f$. By assumption, each nonempty $K_s$ is a compact convex set. We first note that, for every $t\geq 0$ and every $s>0$,
\[
\lev_{\geq s}(f\star \mathbbm{1}_{tB_2^n})=K_s+tB_2^n.
\]
Indeed,
\[
(f\star \mathbbm{1}_{tB_2^n})(x)=
\sup_{z\in tB_2^n} f(x-z)=
\sup_{y\in x-tB_2^n} f(y).
\]
Since $f$ is upper semicontinuous and $x-tB_2^n$ is compact, this supremum is attained. Hence $(f\star \mathbbm{1}_{tB_2^n})(x)\geq s$ if and only if there exists
$y\in x-tB_2^n$ such that $f(y)\geq s$, which is equivalent to
$x\in K_s+tB_2^n$. Thus, by the layer-cake formula,
\[
F(t)=\int_{\R^n}(f\star \mathbbm{1}_{tB_2^n})(x)\,dx
=\int_0^\infty \vol_n(K_s+tB_2^n)\,ds.
\]
In particular,
\[
F(t)-F(0)=\int_0^\infty
\left[\vol_n(K_s+tB_2^n)-\vol_n(K_s)\right]\,ds .
\]

Fix $t_0>0$ such that $F(t_0)<\infty$. For $f\in \LC_n$, such a $t_0$ exists; indeed, every $t_0>0$ works, because $f\star \mathbbm{1}_{t_0 B_2^n}$ is again a coercive log-concave
function, and hence is integrable. For $0<t\leq t_0$, define the difference quotient
\[
q_t(s):=\frac{\vol_n(K_s+tB_2^n)-\vol_n(K_s)}{t}.
\]
Then we have
\[
\frac{F(t)-F(0)}{t}=\int_0^\infty q_t(s)\,ds.
\]
For each fixed $s>0$, Steiner's formula yields
\[
\vol_n(K_s+tB_2^n)=\vol_n(K_s)+S(K_s)t+\sum_{j=2}^n a_j(K_s)t^j,
\]
where the coefficients $a_j(K_s)$ are nonnegative. Therefore,
\[
q_t(s)=S(K_s)+\sum_{j=2}^n a_j(K_s)t^{j-1}.
\]
It follows that, for each fixed $s$, the function $t\mapsto q_t(s)$ is nondecreasing on $(0,\infty)$ and $\lim_{t\to 0+}q_t(s)=S(K_s)$. Consequently, for every $0<t\leq t_0$, we have $0\leq q_t(s)\leq q_{t_0}(s)$.

It remains to check that $q_{t_0}$ is integrable. Since $F(t_0)<\infty$ and $F(0)\leq F(t_0)$, we have
\[
\int_0^\infty q_{t_0}(s)\,ds
=\frac{1}{t_0}\int_0^\infty\left[\vol_n(K_s+t_0B_2^n)-\vol_n(K_s)
\right]\,ds
=\frac{F(t_0)-F(0)}{t_0}<\infty.
\]
Thus, the family $\{q_t:0<t\leq t_0\}$ is dominated by the integrable function $q_{t_0}$. By the dominated convergence theorem,
\[
\lim_{t\to 0+}\frac{F(t)-F(0)}{t}
=\lim_{t\to 0+}\int_0^\infty q_t(s)\,ds
=\int_0^\infty \lim_{t\to 0+}q_t(s)\,ds
=\int_0^\infty S(K_s)\,ds.
\]
This proves that $F$ is right-differentiable at $0$ and that
\[
\left.\frac{dF}{dt}\right|_{t=0+}
=\int_0^\infty S(\lev_{\geq s}f)\,ds.
\]
Finally, if $f\in \LC_n$, then $t\cdot \mathbbm{1}_{B_2^n}=\mathbbm{1}_{tB_2^n}$. Hence, by the definition of the first variation,
\[
W_1(f)=\delta(f,\mathbbm{1}_{B_2^n})
=\lim_{t\to 0+}\frac{\int_{\mathbb R^n}(f\star \mathbbm{1}_{tB_2^n})(x)\,dx-\int_{\mathbb R^n}f(x)\,dx}{t}
=\left.\frac{dF}{dt}\right|_{t=0+}.
\]
Therefore,
\[
W_1(f)=\int_0^\infty S(\lev_{\geq s}f)\,ds.
\]\qed

\subsection{Proof of Proposition \ref{prop:proj-body-LCn}}

    Let $f\in\LC_n$. Since $h_{\Pi_\sharp^b f}$ is finite, even, and 1-homogeneous by \eqref{fn-proj-body-2}, it is the support function of an origin-symmetric convex body $\Pi_\sharp^b f$. Therefore $\Pi_\sharp f=\mathbbm{1}_{\Pi_\sharp^b f}$ 
    is an even indicator of a convex body, so $\Pi_\sharp f\in\LC_n^e$. Moreover, by definition, $\Pi_\sharp^\circ f=\exp(-h_{\Pi_\sharp^b f})$, and 
    since $h_{\Pi_\sharp^b f}$ is even, $\Pi_\sharp^\circ f$ is even as well.  It remains to show that $\Pi_\sharp^\circ f\in\LC_n$. 

    The function $h_{\Pi_\sharp^b f}=h_{\Pi\langle f\rangle}$ is convex as a support function; we show this directly. For every $x,y\in\R^n$ and $t\in[0,1]$, by the triangle inequality,
    \begin{align*}
        h_{\Pi_\sharp^b f}(tx+(1-t)y)
        &=\frac{1}{2}\int_{\R^n}|\langle z,tx+(1-t)y\rangle|\,d\mu_f(z)+\frac{1}{2}\int_{\Sp}|\langle \theta,tx+(1-t)y\rangle|\,d\nu_f(\theta)\\
        &=\frac{1}{2}\int_{\R^n}|\langle z,tx\rangle+\langle z,(1-t)y\rangle|\,d\mu_f(z)
        +\frac{1}{2}\int_{\Sp}|\langle \theta,tx\rangle+\langle\theta,(1-t)y\rangle|\,d\nu_f(\theta)\\
        &\leq \frac{1}{2}\int_{\R^n}|\langle z,tx\rangle|\,d\mu_f(z)+\frac{1}{2}\int_{\R^n}|\langle z,(1-t)y\rangle|\,d\mu_f(z)\\
        &\qquad +\frac{1}{2}\int_{\Sp}|\langle \theta,tx\rangle|\,d\nu_f(\theta)+\frac{1}{2}\int_{\Sp}|\langle\theta,(1-t)y\rangle|\,d\nu_f(\theta)\\
        &=th_{\Pi_\sharp^b f}(x)+(1-t)h_{\Pi_\sharp^b f}(y).
    \end{align*}
    This shows that $h_{\Pi_\sharp^b f}$ is convex. Thus, $\Pi_\sharp^\circ f=\exp(-h_{\Pi_\sharp^b f})$ is log-concave.

    Next, we show that $h_{\Pi_\sharp^b f}$ is continuous. Set
    \[
    h_1(x):=\frac{1}{2}\int_{\R^n}|\langle z,x\rangle|\,d\mu_f(z)\quad\text{and}\quad
    h_2(x):=\frac{1}{2}\int_{\Sp}|\langle\theta,x\rangle|\,d\nu_f(\theta),
    \]
    so that $h_{\Pi_\sharp^b f}=h_1+h_2$. For any $x,y\in\R^n$,
    \begin{align*}
        |h_1(x)-h_1(y)|
        &\leq\frac{1}{2}\int_{\R^n}\big||\langle z,x\rangle|-|\langle z,y\rangle|\big|\,d\mu_f(z)\\
        &\leq\frac{1}{2}\int_{\R^n}|\langle z,x-y\rangle|\,d\mu_f(z)\\
        &\leq \frac{1}{2}|x-y|\int_{\R^n}|z|\,d\mu_f(z),
    \end{align*}
    where in the penultimate line we used $\big||a|-|b|\big|\leq|a-b|$ for any real numbers $a,b$. Similarly,
    \begin{align*}
        |h_2(x)-h_2(y)|
        &\leq \frac{1}{2}\int_{\Sp}|\langle\theta,x-y\rangle|\,d\nu_f(\theta)\\
        &\leq\frac{1}{2}|x-y|\int_{\Sp}|\theta|\,d\nu_f(\theta)
        = \frac{1}{2}\nu_f(\Sp)|x-y|,
    \end{align*}
    since $|\theta|=1$ for $\theta\in\Sp$. Therefore,
    \begin{align*}
        |h_{\Pi_\sharp^b f}(x)-h_{\Pi_\sharp^b f}(y)|
        &\leq |h_1(x)-h_1(y)|+|h_2(x)-h_2(y)|\\
        &\leq\frac{1}{2}\left(\int_{\R^n}|z|\,d\mu_f(z)+\nu_f(\Sp)\right)|x-y|.
    \end{align*}
    Since $\mu_f$ has finite first moment and $\nu_f$ is finite, this proves that $h_{\Pi_\sharp^b f}$ is Lipschitz, hence continuous.

    We now show that $h_{\Pi_\sharp^b f}$ is coercive. We claim that $h_{\Pi_\sharp^b f}(u)>0$ for every $u\in\Sp$. Indeed, if $h_{\Pi_\sharp^b f}(u)=0$, then both nonnegative integrands in \eqref{fn-proj-body-2} vanish almost everywhere, so
    \[
    \langle z,u\rangle=0\quad \mu_f\text{-a.e.}
    \qquad\text{and}\qquad
    \langle \theta,u\rangle=0\quad \nu_f\text{-a.e.}
    \]
    Hence, $\mu_f$ and $\nu_f$ are both supported on the same hyperplane
    \[
    u^\perp=\{x\in\R^n:\,\langle x,u\rangle=0\},
    \]
    contradicting Theorem \ref{thm:rotem-falah-main}(iii). Therefore $h_{\Pi_\sharp^b f}(u)>0$ for all $u\in\Sp$. Let
    \[
    c_f:=\min_{u\in\Sp}h_{\Pi_\sharp^b f}(u).
    \]
    By the continuity of $h_{\Pi_\sharp^b f}$ and the compactness of $\Sp$, the minimum exists and is positive. Thus for every $x\in\R^n\setminus\{o\}$, by the 1-homogeneity of the support function we get
    \begin{equation}\label{eq:proj-body-supp-coercive}
        h_{\Pi_\sharp^b f}(x)=|x|\,h_{\Pi_\sharp^b f}\!\left(\frac{x}{|x|}\right)\geq c_f|x|.
    \end{equation}
  This shows that $h_{\Pi_\sharp^b f}$ is coercive. Hence,
    \[
    \Pi_\sharp^\circ f(x)=\exp(-h_{\Pi_\sharp^b f}(x))\leq \exp(-c_f|x|).
    \]
   This completes the proof  that $\Pi_\sharp^\circ f\in\LC_n^e$. \qed 

\subsection{Proof of Theorem \ref{thm:Blaschke=convergence-thm}}

We divide the proof into six steps.

\smallskip
\noindent
\emph{Step 1: The averaging operator on surface area measures.} For $u\in\Sp$, define the \emph{averaging operator}
\[
T_u(\mu,\nu):=\frac{1}{2}\bigl(\mu+(R_u)_\#\mu,\nu+(R_u)_\#\nu\bigr).
\]
By Definition~\ref{def:mainDef},
\[
(\mu_{B_u^\sharp f},\nu_{B_u^\sharp f})=T_u(\mu_f,\nu_f).
\]
Hence, if $f_k:=B_{u_k}^\sharp\cdots B_{u_1}^\sharp f$, then
\[
(\mu_{f_k},\nu_{f_k})=T_{u_k}\cdots T_{u_1}(\mu_f,\nu_f).
\]
It will be convenient to rewrite this in a group-theoretic form. For $u\in\Sp$, let
\[
\eta_u:=\frac{1}{2}(\delta_{\rm Id}+\delta_{R_u})
\]
be a probability measure on $\mathrm{O}(n)$, and for a pair $(\mu,\nu)$ set
\[
\eta_u*(\mu,\nu):=
\left(
\int_{\mathrm{O}(n)}U_\#\mu\,d\eta_u(U),
\int_{\mathrm{O}(n)}U_\#\nu\,d\eta_u(U)
\right).
\]
Then $T_u(\mu,\nu)=\eta_u*(\mu,\nu)$.

Define recursively
\[
\alpha_k:=\eta_{u_k}*\cdots *\eta_{u_1},
\]
where the convolution on the right is a convolution of probability measures on $\mathrm{O}(n)$. A straightforward induction on $k$ shows that
\begin{equation}\label{eq:alpha-representation}
(\mu_{f_k},\nu_{f_k})
=
\left(
\int_{\mathrm{O}(n)}U_\#\mu_f\,d\alpha_k(U),\;
\int_{\mathrm{O}(n)}U_\#\nu_f\,d\alpha_k(U)
\right).
\end{equation}
Indeed, for $k=1$ this is just the definition of $T_{u_1}$. If \eqref{eq:alpha-representation}
holds for $k$, then
\[
(\mu_{f_{k+1}},\nu_{f_{k+1}})
=
T_{u_{k+1}}(\mu_{f_k},\nu_{f_k})
=
\eta_{u_{k+1}}*(\mu_{f_k},\nu_{f_k}),
\]
and inserting the induction hypothesis yields the case $k+1$.

\smallskip
\noindent\emph{Step 2: Choice of the reflection sequence.} 
It is well-known that every compact connected Lie group contains a finitely generated dense subgroup (see, e.g., \cite{HofmannMorris1992}). Applying this to $\mathrm{SO}(n)$,
choose elements $A_1,\ldots,A_\ell\in\mathrm{SO}(n)$ which generate a dense subgroup of $\mathrm{SO}(n)$. By the Cartan--Dieudonn\'e theorem, each $A_i\in \mathrm{SO}(n)$ is a product of finitely many hyperplane reflections. Since $\det(A_i)=1$, each $A_i$ may be written as a product of an even number
of hyperplane reflections:
\[
A_i=R_{a_{i,1}}\cdots R_{a_{i,2m_i}}.
\]
Adjoin all of these reflections, together with one additional hyperplane reflection  $R_{a_0}$. The subgroup generated by these reflections has closure containing the dense subgroup
generated by $A_1,\ldots,A_\ell$ in $\mathrm{SO}(n)$, and also contains an orientation-reversing element. Hence its closure is all of $\mathrm{O}(n)$. Relabeling these reflections as $R_{u_1},\ldots,R_{u_m}\in \mathrm{O}(n)$, we obtain reflections whose generated subgroup is dense in $\mathrm{O}(n)$. We now extend this to an infinite sequence by periodicity:
\[
u_{qm+j}:=u_j,\qquad q\geq 0,\quad 1\leq j\leq m.
\]

For $j\in\{1,\dots,m\}$, set
\[
\eta_j:=\eta_{u_j}=\frac{1}{2}(\delta_{\rm Id}+\delta_{R_{u_j}}),
\qquad
\beta:=\eta_m*\cdots *\eta_1.
\]
The support of $\beta$ contains ${\rm Id}$ and each reflection $R_{u_j}$. Indeed, in the convolution product defining $\beta$, choosing $R_{u_j}$ in the $j$th factor and ${\rm Id}$ in all others produces
the point $R_{u_j}$, while choosing ${\rm Id}$ in all factors produces ${\rm Id}$. Hence, the closed subgroup generated by $\operatorname{supp}(\beta)$ is all of $\mathrm{O}(n)$.

We claim that $\beta$ is adapted and strictly aperiodic to apply the Kawada--It\^o theorem (see \cite{Kawada-Ito}, or \cite[Theorem 4.6.3]{Applebaum2014} for a modern account). \emph{Adaptedness} means that the closed subgroup generated by $\operatorname{supp}(\beta)$ is the whole
group, which we already proved. Recall that a probability measure on a compact group is \emph{strictly aperiodic} if its support is not contained in a proper closed left coset. Since ${\rm Id}\in\operatorname{supp}(\beta)$, if $\operatorname{supp}(\beta)\subset gH$ for some closed subgroup
$H$, then ${\rm Id}\in gH$, hence $g\in H$, so $gH=H$. Thus, $\operatorname{supp}(\beta)$ would be
contained in a proper closed subgroup, contradicting adaptedness. Therefore, $\beta$ is adapted and strictly aperiodic, so by the Kawada--It\^o theorem,
\[
\beta^{*q}\xrightarrow[q\to\infty]{}\eta
\]
weakly as probability measures on $\mathrm{O}(n)$, where $\eta$ denotes the Haar probability measure.

Now let $k=qm+r$ with $0\leq r\leq m-1$, and define
\[
\beta_r:=\eta_r*\cdots *\eta_1,
\]
with the convention $\beta_0:=\delta_{\rm Id}$. Then by the periodicity of the sequence $\{u_j\}$, we have
\[
\alpha_k=\beta_r*\beta^{*q}.
\]
Since left convolution by a fixed probability measure is weakly continuous and $\beta_r*\eta=\eta$
for every $r$ (since $\eta$ is the Haar measure), it follows that
\[
\alpha_k\xrightarrow[k\to\infty]{}\eta
\]
weakly on $\mathrm{O}(n)$.

\smallskip
\noindent
\emph{Step 3: The Haar-averaged pair is admissible.} Define
\[
\bar\mu:=\int_{\mathrm{O}(n)}U_\#\mu_f\,d\eta(U),
\qquad
\bar\nu:=\int_{\mathrm{O}(n)}U_\#\nu_f\,d\eta(U).
\]
We verify that $(\bar\mu,\bar\nu)$ satisfies the hypotheses of Theorem~\ref{thm:rotem-falah-main}. First, note that $\bar\mu\not\equiv 0$ since
\[
\bar\mu(\R^n)=\int_{\mathrm{O}(n)}\mu_f(\R^n)\,d\eta(U)=\mu_f(\R^n)
=\int_{\R^n}f(x)\,dx>0.
\]
Second, $\bar\mu$ has finite first moment:
\[
\int_{\R^n}|x|\,d\bar\mu(x)
=
\int_{\mathrm{O}(n)}\int_{\R^n}|Ux|\,d\mu_f(x)\,d\eta(U)
=
\int_{\R^n}|x|\,d\mu_f(x)
<\infty.
\]
Third, $\bar\mu+\bar\nu$ is centered. Indeed, for every $\theta\in\Sp$, since $(\mu_f,\nu_f)$ is centered we have
\begin{align*}
\int_{\R^n}\langle x,\theta\rangle\,d\bar\mu(x)
&+\int_{\Sp}\langle y,\theta\rangle\,d\bar\nu(y)
=\int_{\mathrm{O}(n)}\left[
\int_{\R^n}\langle Ux,\theta\rangle\,d\mu_f(x)
+\int_{\Sp}\langle Uy,\theta\rangle\,d\nu_f(y)
\right]d\eta(U)\\
&=\int_{\mathrm{O}(n)}
\left[\int_{\R^n}\langle x,U^{-1}\theta\rangle\,d\mu_f(x)
+\int_{\Sp}\langle y,U^{-1}\theta\rangle\,d\nu_f(y)
\right]d\eta(U)=0
\end{align*}

Fourth, $\bar\mu$ and $\bar\nu$ are not supported on a common hyperplane.
If $\bar\nu\neq 0$, then $\bar\nu$ is $\mathrm{O}(n)$-invariant, so it is a positive multiple of the spherical measure on $\Sp$, and therefore it is not supported on any hyperplane. It remains to consider the case $\bar\nu=0$. Then
\[
0=\bar\nu(\Sp)=\int_{\mathrm{O}(n)}\nu_f(\Sp)\,d\eta(U)=\nu_f(\Sp),
\]
so $\nu_f=0$. Since $(\mu_f,\nu_f)$ satisfies condition (iii) of Theorem~\ref{thm:rotem-falah-main}, $\mu_f$ is not supported on any
hyperplane.

Note that $\bar\mu$ is also $\mathrm{O}(n)$-invariant. If $\bar\mu$ were supported on a hyperplane, then by its $\mathrm{O}(n)$-invariance it would have to be supported on $\{o\}$; equivalently $\bar\mu=c\delta_o$
for some $c>0$. But reflection averaging preserves the radial pushforward, i.e.,
for every bounded Borel function $\psi:[0,\infty)\to\R$, we have 
\[
\int_{\R^n}\psi(|x|)\,d\bar\mu(x)
=\int_{\mathrm{O}(n)}\int_{\R^n}\psi(|Ux|)\,d\mu_f(x)\,d\eta(U)=\int_{\R^n}\psi(|x|)\,d\mu_f(x).
\]
Thus $|x|_\#\bar\mu=|x|_\#\mu_f$. If $\bar\mu=c\delta_o$, then $|x|_\#\mu_f=c\delta_o$, so
$\mu_f$ is supported on $\{o\}$, contradicting the fact that $\mu_f$ is not supported on any hyperplane.
Hence, $\bar\mu$ is not supported on any hyperplane. This proves the admissibility of $(\bar\mu,\bar\nu)$. Therefore, by Theorem~\ref{thm:rotem-falah-main} there exists $g\in\LC_n$, unique up to translation, such that
\begin{equation}\label{eq:g-averaged-pair}
(\mu_g,\nu_g)=(\bar\mu,\bar\nu).
\end{equation}

\smallskip
\noindent
\emph{Step 4: Cosmic convergence of the surface area pairs.} Let $\xi:\R^n\to\R$ be cosmically continuous, and define
\[
\overline{\xi}(\theta):=\lim_{\lambda\to\infty}\frac{\xi(\lambda\theta)}{\lambda},
\qquad \theta\in\Sp.
\]
Since the convergence is uniform in $\theta$, the function $\overline{\xi}$ is continuous on $\Sp$. Moreover, by uniform convergence there exists $R>0$ such that
\[
\left|\frac{\xi(r\theta)}{r}-\overline{\xi}(\theta)\right|\leq 1
\qquad\text{for all }r\geq R,\, \theta\in\Sp.
\]
Hence for $|x|=r\geq R$, we have
\[
|\xi(x)|\leq r\left(\sup_{\theta\in\Sp}|\overline{\xi}(\theta)|+1\right),
\]
and on the compact ball $RB_2^n$, the function $\xi$ is bounded. Thus, there exists $C_\xi>0$ such that
\[
|\xi(x)|\leq C_\xi(1+|x|)
\qquad\text{for all }x\in\R^n.
\]
Now define
\[
F_\xi(U):=
\int_{\R^n}\xi(Ux)\,d\mu_f(x)
+\int_{\Sp}\overline{\xi}(U\theta)\,d\nu_f(\theta),\qquad U\in\mathrm{O}(n).
\]
This is well-defined because $\mu_f$ has finite first moment and $\nu_f$ is finite.
It is also continuous in $U$. Indeed, if $U_j\to U$ in $\mathrm{O}(n)$, then $\xi(U_j x)\to\xi(Ux)$ for every $x$, and $\overline{\xi}(U_j\theta)\to\overline{\xi}(U\theta)$ for every $\theta$; the dominating
functions $C_\xi(1+|x|)$ and $\sup_{\eta\in\Sp}|\overline{\xi}(\eta)|$ will allow us to apply the dominated convergence theorem.

Using \eqref{eq:alpha-representation}, we obtain
\begin{align*}
\int_{\R^n}\xi\,d\mu_{f_k}+\int_{\Sp}\overline{\xi}\,d\nu_{f_k}
&=
\int_{\mathrm{O}(n)}F_\xi(U)\,d\alpha_k(U).
\end{align*}
Since $\alpha_k\to\eta$ weakly and $F_\xi$ is continuous and bounded on the compact group $\mathrm{O}(n)$, we deduce that
\begin{align*}
\int_{\R^n}\xi\,d\mu_{f_k}+\int_{\Sp}\overline{\xi}\,d\nu_{f_k}
&\longrightarrow
\int_{\mathrm{O}(n)}F_\xi(U)\,d\eta(U)
=\int_{\R^n}\xi\,d\bar\mu+\int_{\Sp}\overline{\xi}\,d\bar\nu.
\end{align*}
By the definition of cosmic convergence (see \cite[Definition 1.8]{Falah-Rotem}), this means that $(\mu_{f_k},\nu_{f_k})\to (\bar\mu,\bar\nu)$ cosmically.

\smallskip
\noindent
\emph{Step 5: Convergence up to translation.} By \eqref{eq:g-averaged-pair}, the limiting pair of surface area measures is $(\mu_g,\nu_g)$. Hence $(\mu_{f_k},\nu_{f_k})\to(\mu_g,\nu_g)$ cosmically. By \cite[Theorem 1.10]{Falah-Rotem}, there exist translations $v_k\in\R^n$ such that $\widetilde{f}_k(x):=f_k(x+v_k)$ converges to $g$ in $\LC_n$. Writing $\widetilde{f}_k=e^{-\phi_k}$ and $g=e^{-\phi}$, this means that $\phi_k$ epi-converges to $\phi$. Since the function $t\mapsto e^{-t}$ is continuous and strictly decreasing, this is equivalent to the hypo-convergence of $\widetilde f_k$ to $g$.

\smallskip
\noindent
\emph{Step 6: The limit can be chosen to be radial.} For every $U\in\mathrm{O}(n)$, the pair $(\bar\mu,\bar\nu)$ is $U$-invariant, so
\[
(\mu_{g\circ U^{-1}},\nu_{g\circ U^{-1}})
=(U_\#\bar\mu,U_\#\bar\nu)
=(\bar\mu,\bar\nu)
=(\mu_g,\nu_g).
\]
By the uniqueness in Theorem~\ref{thm:rotem-falah-main}, there exists $a_U\in\R^n$ such that
\[
\forall x\in\R^n,\quad g(U^{-1}x)=g(x+a_U).
\]
Choose $t\in(0,\|g\|_\infty)$. Since $g\in\LC_n$, the superlevel set $K_t:=\lev_{\geq t}g$ is a nonempty compact convex body. The identity above implies that $UK_t = K_t-a_U$. Let $s(K_t)$ denote the Steiner point of $K_t$. Using the covariance of the Steiner point, $s(UK_t)=Us(K_t)$ and $s(K_t-a_U)=s(K_t)-a_U$, so we obtain $Us(K_t)=s(K_t)-a_U$. Thus, with $b:=s(K_t)$, we get $a_U=b-Ub$. Now define $f^\sharp(x):=g(x+b)$. 
Then for every $U\in\mathrm{O}(n)$ and every $x\in\R^n$, we have
\begin{align*}
f^\sharp(U^{-1}x)
&=g(U^{-1}x+b) =g(U^{-1}(x+Ub))
 =g(x+Ub+a_U)=g(x+b)
 =f^\sharp(x),
\end{align*}
since $a_U=b-Ub$. Hence, $f^\sharp$ is radial.

Finally, replacing each vector $v_k$ by $v_k+b$, the sequence $\widetilde{f}_k(x)=f_k(x+v_k+b)$
 hypo-converges to $f^\sharp$. Since translations do not change the surface area pair, $f^\sharp$ satisfies
\[
(\mu_{f^\sharp},\nu_{f^\sharp})
=(\bar\mu,\bar\nu)
=\left(
\int_{\mathrm{O}(n)}U_\#\mu_f\,d\eta(U),\;
\int_{\mathrm{O}(n)}U_\#\nu_f\,d\eta(U)
\right).
\]
This completes the proof. \qed
\subsection{Proof of Lemma \ref{lem:as-hypo-symmetral}}

Let $f\in\LC_n$. Recall that $f^\sharp$ is radial, i.e., for every $\vartheta\in\mathrm{O}(n)$ we have $f^\sharp\circ\vartheta=f^\sharp$. Using  Proposition \ref{prop:affine-sa}(i) with $\det(\vartheta)=\pm 1$, we get that for all $Q\in\mathcal{S}_c^n$ and all $\vartheta\in\mathrm{O}(n)$,
\[
\delta(f^\sharp,\mathbbm{1}_{Q^\circ})=\delta(f^\sharp\circ\vartheta,\mathbbm{1}_{Q^\circ})=\delta(f^\sharp,\mathbbm{1}_{\vartheta Q^\circ}).
\]
Averaging over $\mathrm{O}(n)$ and applying Theorem \ref{thm:rotem-falah-variation} with $g=\mathbbm{1}_{\vartheta Q^\circ}$ (so that $h_g=h_{\vartheta Q^\circ}$ and $\supp(g)=\vartheta Q^\circ$), we obtain 
\begin{align*}
\delta(f^\sharp,\mathbbm{1}_{Q^\circ})&=\int_{\mathrm{O}(n)}\delta(f^\sharp,\mathbbm{1}_{\vartheta Q^\circ})\,d\eta(\vartheta)\\
&=\int_{\mathrm{O}(n)}\left(\int_{\R^n}h_{\vartheta Q^\circ}(x)\,d\mu_{f^\sharp}(x)+\int_{\Sp}h_{\vartheta Q^\circ}(u)\,d\nu_{f^\sharp}(u)\right)d\eta(\vartheta)\\
&= \int_{\R^n}\left(\int_{\mathrm{O}(n)}h_{\vartheta Q^\circ}(x)\,d\eta(\vartheta)\right)d\mu_{f^\sharp}(x)+\int_{\Sp}\left(\int_{\mathrm{O}(n)}h_{\vartheta Q^\circ}(u)\,d\eta(\vartheta)\right)d\nu_{f^\sharp}(u).
\end{align*} 
In the last line, we interchanged the order of integration using Tonelli's theorem. Let $x\in\R^n$. Since $h_{\vartheta Q^\circ}(x)=h_{Q^\circ}(\vartheta^\top x)$, averaging over $\vartheta\in\mathrm{O}(n)$ gives us
\[
\int_{\mathrm{O}(n)}h_{Q^\circ}(\vartheta^\top x)\,d\eta(\vartheta)=\frac{|x|}{n\omega_n}\int_{\Sp}h_{Q^\circ}(v)\,d\sigma(v),
\]
because $\vartheta^\top(x/|x|)$ is uniformly distributed on $\Sp$. Likewise, for $u\in\Sp$,
\[
\int_{\mathrm{O}(n)}h_{Q^\circ}(\vartheta^\top u)\,d\eta(\vartheta)=\frac{1}{n\omega_n}\int_{\Sp}h_{Q^\circ}(v)\,d\sigma(v).
\]

Substituting these integrals back into the last expression for $\delta(f^\sharp,\mathbbm{1}_{Q^\circ})$, we obtain
\begin{equation}\label{eq:first-variation-mean-symmetral}
\begin{split}
   &\delta(f^\sharp,\mathbbm{1}_{Q^\circ})\\&=  \int_{\R^n}\left(\frac{|x|}{n\omega_n}\int_{\Sp}h_{Q^\circ}(v)\,d\sigma(v)\right)d\mu_{f^\sharp}(x)+\int_{\Sp}\left(\frac{1}{n\omega_n}\int_{\Sp}h_{Q^\circ}(w)\,d\sigma(w)\right)d\nu_{f^\sharp}(w)\\
  &=\frac{1}{n\omega_n}\left(\int_{\R^n}|x|\,d\mu_{f^\sharp}(x)+\nu_{f^\sharp}(\Sp)\right)\left(\int_{\Sp}h_{Q^\circ}(v)\,d\sigma(v)\right)\\
  &=\frac{W_1(f^\sharp)}{n\omega_n}\int_{\Sp}h_{Q^\circ}(v)\,d\sigma(v), 
\end{split}  
\end{equation}
where the last line follows from Remark \ref{rmk:SA}.

Next, we apply H\"older's inequality to the functions $a(v):=h_{Q^\circ}(v)^{\frac{n}{n+1}}$ and $b(v):=h_{Q^\circ}(v)^{-\frac{n}{n+1}}$ with $p=\frac{n+1}{n}>1$ and $q=n+1>1$, so that $\frac{1}{p}+\frac{1}{q}=1$:
\begin{align*}
&\left(\int_{\Sp}a(v)^{\frac{n+1}{n}}\,d\sigma(v)\right)^{\frac{n}{n+1}}\left(\int_{\Sp}b(v)^{n+1}\,d\sigma(v)\right)^{\frac{1}{n+1}} \\
&\geq \int_{\Sp}a(v)b(v)\,d\sigma(v)=\int_{\Sp}1\,d\sigma(v)=\sigma(\Sp)=n\omega_n.
\end{align*}
In other words,
\begin{equation}\label{eq:holder-main}
    \left(\int_{\Sp}h_{Q^\circ}(v)\,d\sigma(v)\right)^{\frac{n}{n+1}}\left(\int_{\Sp}h_{Q^\circ}(v)^{-n}\,d\sigma(v)\right)^{\frac{1}{n+1}} \geq n\omega_n.
\end{equation}
Since $h_{Q^\circ}^{-1}=\rho_Q$, using polar coordinates we get $\int_{\Sp}h_{Q^\circ}(v)^{-n}\,d\sigma(v)=n\vol_n(Q)$. Hence, 
\begin{equation}\label{eq:Holder-2}
\forall Q\in\mathcal{S}_c^n,\quad     (n\vol_n(Q))^{\frac{1}{n}}\int_{\Sp}h_{Q^\circ}(v)\,d\sigma(v)\geq (n\omega_n)^{\frac{n+1}{n}}.
\end{equation}
Therefore, substituting \eqref{eq:Holder-2} into \eqref{eq:first-variation-mean-symmetral} and using Proposition \ref{prop:hypo-symmetral-preserves-SA},  we deduce that for all $Q\in\mathcal{S}_c^n$,
\begin{align*}
\delta(f^\sharp,\mathbbm{1}_{Q^\circ}) &=\frac{W_1(f^\sharp)}{n\omega_n}\cdot\frac{(n\vol_n(Q))^{\frac{1}{n}}}{n^{\frac{1}{n}}}\int_{\Sp}h_{Q^\circ}(v)\,d\sigma(v)\geq  W_1(f^\sharp)\omega_n^{\frac{1}{n}}=W_1(f)\omega_n^{\frac{1}{n}}.
\end{align*}

Finally, we analyze the equality conditions. The only inequality in the entire proof is \eqref{eq:holder-main}, where we used H\"older's inequality. Thus, equality holds if and only if $h_{Q^\circ}$ is a constant almost everywhere on the sphere $\Sp$ (hence everywhere on $\Sp$ by the continuity of $h_{Q^\circ}$). Therefore, $Q^\circ$ is a Euclidean ball, so $Q=(Q^\circ)^\circ$ is also a Euclidean ball. Conversely, if $Q=rB_2^n$ for some $r>0$, then $Q^\circ=r^{-1}B_2^n$, so  $\delta(f^\sharp,\mathbbm{1}_{Q^\circ})=r^{-1}W_1(f^\sharp)$ and $\vol_n(Q)=r\omega_n^{\frac{1}{n}}$. Hence $\delta(f^\sharp,\mathbbm{1}_{Q^\circ})\vol_n(Q)^{\frac{1}{n}}=W_1(f^\sharp)\omega_n^{\frac{1}{n}}$, so equality holds. Thus,
\[
\inf_{Q\in\mathcal{S}_c^n}\delta(f^\sharp,\mathbbm{1}_{Q^\circ})\vol_n(Q)^{\frac{1}{n}}=\omega_n^{\frac{1}{n}}W_1(f^\sharp)=\omega_n^{\frac{1}{n}}W_1(f)
\]
with minimizers $Q=rB_2^n$. Therefore,
\[
\Omega_\sharp(f^\sharp)=n^{\frac{1}{n+1}}\left(\omega^{\frac{1}{n}}W_1(f^\sharp)\right)^{\frac{n}{n+1}}=(n\omega_n)^{\frac{1}{n+1}}W_1(f^\sharp)^{\frac{n}{n+1}}=(n\omega_n)^{\frac{1}{n+1}}W_1(f)^{\frac{n}{n+1}}
\]
where we again used Proposition \ref{prop:hypo-symmetral-preserves-SA}. This completes the proof. \qed
\section*{Acknowledgments}

We would like to thank Michael Roysdon for the discussions related to this paper.
\bibliographystyle{plain}
\bibliography{main}


\vspace{3mm}

\noindent {\sc Department of Mathematics \& Computer Science, Longwood University, Farmville, Virginia 23909}

\noindent {\it E-mail address:} {\tt hoehnersd@longwood.edu}

\bigskip

\noindent {\sc Department of Mathematics, Applied Mathematics and Statistics, Case Western Reserve University, Cleveland, Ohio 44106}

\noindent {\it E-mail address:} {\tt exc583@case.edu}

\end{document}